\documentclass{article}
\usepackage[utf8]{inputenc}
\usepackage{amsmath}
\usepackage{amsfonts}
\usepackage{amsthm}
\usepackage{color}
\usepackage{graphicx}
\usepackage{epstopdf}

\newcommand{\eps}{\epsilon}

\newtheorem{theorem}{Theorem}[section]
\newtheorem{proposition}[theorem]{Proposition}
\newtheorem{lemma}[theorem]{Lemma}

\newtheorem{definition}{Definition}[section]

\title{On the lifting of deterministic convergence rates for inverse problems with stochastic noise}
\author{Daniel Gerth\footnote{Corresponding author. Faculty of Mathematics, TU Chemnitz, 09107 Chemnitz, Germany. E-mail: daniel.gerth@mathematik.tu-chemnitz.de. Research supported in part by the Austrian Science Fund (FWF): W1214-N15 and by the German Research Foundation (DFG) under grant HO 1454/8-2.}, Andreas Hofinger\footnote{Johann Radon Institute, Altenbergerstrasse 69, 4040 Linz, Austria}, Ronny Ramlau\footnote{Industrial Mathematics Institue, JKU Linz, Altenbergerstrasse 69, 4040 Linz, Austria; Johann Radon Institute, Altenbergerstrasse 69, 4040 Linz, Austria}}

\begin{document}
\maketitle
\begin{abstract}
Both for the theoretical and practical treatment of Inverse Problems, the modeling of the noise is a crucial part. One either models the measurement via a deterministic worst-case error assumption or assumes a certain stochastic behavior of the noise. Although some connections between both models are known, the communities develop rather independently. In this paper we seek to bridge the gap between the deterministic and the stochastic approach and show convergence and convergence rates for Inverse Problems with stochastic noise by lifting the theory established in the deterministic setting into the stochastic one. This opens the wide field of deterministic regularization methods for stochastic problems without having to do an individual stochastic analysis for each problem.
\end{abstract}

In Inverse Problems, the model of the inevitable data noise is of utmost importance. In most cases, an additive noise model
\begin{equation}\label{eq:ad_noise}
y^\text{noisy}=y+\epsilon
\end{equation}
is assumed. In \eqref{eq:ad_noise}, $y\in\mathcal{Y}$ is the true data of the unknown $x\in\mathcal{X}$ under the action of the (in general) nonlinear operator $F:\mathcal{X}\rightarrow\mathcal{Y}$,
\begin{equation}\label{eq:problem}
F(x)=y,
\end{equation}
and $\epsilon$ in \eqref{eq:ad_noise} corresponds to the noise. The spaces $\mathcal{X},\mathcal{Y}$ are assumed to be Banach- or Hilbert spaces. When speaking of Inverse Problems, we assume that \eqref{eq:problem} is ill-posed. In particular this means that solving \eqref{eq:problem} for $x$ with noisy data \eqref{eq:ad_noise} is unstable in the sense that ``small'' errors in the data may lead to arbitrarily large errors in the solution. Hence, \eqref{eq:ad_noise} is not a sufficient description of the noise. More information is needed in order to compute solutions from the data in a stable way. In the \textit{deterministic} setting, one assumes 
\begin{equation}\label{eq:det_delta}
d_\mathcal{Y}(y,y^\delta)\leq\delta
\end{equation} 
for some $\delta>0$ where $d_\mathcal{Y}(\cdot,\cdot)$ is an appropriate distance functional. Typically, $d_\mathcal{Y}$ is induced by a norm such that \eqref{eq:det_delta} reads $||y-y^\delta||\leq \delta$. Here and further on we use the superscript $\cdot^\delta$ to indicate the deterministic setting. Solutions of \eqref{eq:problem} under the assumption \eqref{eq:ad_noise},\eqref{eq:det_delta} are often computed via a Tikhonov-type variational approach 
\begin{equation}\label{eq:functional}
x_\alpha^\delta=\min_{x\in\mathcal{D}(F)} d_\mathcal{Y}(F(x),y^\delta)+\alpha \Phi(x)
\end{equation}
where again $d_\mathcal{Y}$ is a distance function and $\Phi(\cdot)$ is the penalty term used to stabilize the problem and to incorporate a-priori knowledge into the solution. The regularization parameter $\alpha$ is used to balance between data misfit and the penalty and has to be chosen appropriately. The literature in the deterministic setting is rich, at this point we only refer to the monographs \cite{Louis,EHN1996,Hof86} for an overview.

 The deterministic worst-case error stands in contrast to \textit{stochastic} noise models where a certain distribution of the noise $\epsilon$ in \eqref{eq:ad_noise} is assumed. We shall indicate the stochastic setting by the superscript $\cdot^\eta$. In this paper, $\eta$ will be the parameter controlling the variance of the noise. Depending on the actual distribution of $\epsilon$, $d_\mathcal{Y}(y,y^\eta)$ may be arbitrarily large, but with low probability. A very popular approach to find a solution of \eqref{eq:problem} is the Bayesian method. For more detailed information, we refer to \cite{KaipioSomersalo,stuart_bayes,MoseSam,tarantola,CalSom}. In the Bayesian setting, the solution of the Inverse Problem is given as a distribution of the random variable of interest, the \textit{posterior distribution} $\pi_{post}$, determined by Bayes formula
\begin{equation}\label{eqn:bayes}
\pi_{post}(x|y^\eta)=\frac{\pi_{\epsilon}(y^\eta | x)\pi_{pr}(x)}{\pi_{y^\eta}(y^\eta)}.
\end{equation}
That is, roughly spoken, all values $x$ are assigned a probability of being a solution to \eqref{eq:problem} given the noisy data $y^\eta$. In \eqref{eqn:bayes}, the \textit{likelihood function} $\pi_{\epsilon}(y^\eta | x)$ represents the model for the measurement noise whereas the \textit{prior distribution} $\pi_{pr}(x)$ represents a-priori information about the unknown. The data distribution $\pi_{y^\eta}(y^\eta)$ as well as the normalization constants are usually neglected since they only influence the normalization of the posterior distribution. In practice one is often more interested in finding a single representation as solution instead of the distribution itself. Popular point estimates are the \textit{conditional expectation} (conditional mean, CM) 
\begin{equation}\label{eq:cm}
\mathbb{E}(\pi_{post}(x|y^\eta))=\int x \pi_{post}(x|y^\eta) dx
\end{equation}
and the \textit{maximum a-posteriori (MAP)} solution 
\begin{equation}\label{def:xmap}
x^{\mathrm{MAP}}=\underset{x}{\mathrm{argmax}}\quad \pi_{post}(x|y^\eta),
\end{equation}
i.e., the most likely value for $x$ under the prior distribution given the data $y^\eta$. Both point estimators are widely used. The computation of the CM-solution is often slow since it requires repeated sampling of stochastic quantities and the evaluation of high-dimensional integrals. The MAP-solution, however, essentially leads to a Tikhonov-type problem. Namely, assuming $\pi_{\epsilon}(y^\eta | x)\propto\exp(-d_\mathcal{Y}(F(x),y^\eta))$ and $\pi_{pr}(x)\propto\exp(-\alpha \Phi(x))$, one has
\begin{align}
x^{\mathrm{MAP}}&=\underset{x}{\mathrm{argmax}} \exp(-d_\mathcal{Y}(F(x),y^\eta))\exp(-\alpha \Phi(x))\nonumber\\
&=\underset{x}{\mathrm{argmin}}\,d_\mathcal{Y}(F(x),y^\eta)+\alpha \Phi(x)\nonumber
\end{align}
analogously to \eqref{eq:functional}.

 Also non-Bayesian approaches for Inverse Problems often seek to minimize a functional \eqref{eq:functional}, see e.g. \cite{HohWer12,BisHohMunk04} or use techniques known from deterministic theory such as filter methods \cite{BlanchardMathe,BisHoh07}. Finally, Inverse Problems appear in the context of statistics. Hence, the statistics community has developed methods to solve \eqref{eq:problem}, partly again based on the minimization of \eqref{eq:functional}. We refer to \cite{EvStark} for an overview.

In summary, Tikhonov-type functionals \eqref{eq:functional} and other deterministic methods frequently appear also in the stochastic setting. From a practical point of view, one would expect to be able to use deterministic regularization methods for \eqref{eq:problem} even when the noise is stochastic. Indeed, the main question for the actual computation of the solution, given a particular sample of noisy data $y^\eta$, is the choice of the regularization parameter. A second question, mostly coming from the deterministic point of view, is the one of convergence of the solutions when the noise approaches zero. In the stochastic setting these questions are answered often by a full stochastic analysis of the problem. In this paper we present a framework that allows to find appropriate regularization parameters, prove convergence of regularization methods and find convergence rates for Inverse Problems with a stochastic noise model by directly using existing results from the deterministic theory. 

The paper takes several ideas from the dissertation \cite{HOF06}, which is only publicly available as book \cite{Hof_diss_book} and not published elsewhere. It is organized as follows. In Section \ref{ssec:noisemodel} we discuss an issue occurring in the transition from deterministic to stochastic noise for infinite dimensional problems. The Ky Fan metric, which will be the main ingredient of our analysis, and its relation to the expectation will be introduced in Section \ref{ssec:kyfan}. We present our framework to lift convergence results from the deterministic setting into the stochastic setting in Section \ref{ssec:conv_stoch}. Examples for the lifting strategy are given in Section \ref{sec:examples}.

\section{On the noise model}\label{ssec:noisemodel}
Before addressing the convergence theory, we would like to discuss stochastic noise modeling and its intrinsic conflict with the deterministic model. Here and throughout the rest of the paper, assume 
\begin{equation}\label{def:probspace}
(\Omega,\mathcal{F},\mathbb{P})
\end{equation}
to be a complete probability space with a set $\Omega$ of outcomes of the stochastic event, $\mathcal{F}$ the corresponding $\sigma$-algebra and $\mathbb{P}$ a probability measure, $\mathbb{P}:(\Omega,\mathcal{F})\rightarrow[0,1]$. We restrict ourselves here to probability measures for the sake of simplicity. Extensions to more general measures are straightforward. In the Hilbert-space setting, the noise is typically modeled as follows, see for example \cite{LasSakSil09,Louis,BisHoh07,Hanne14}. Let $\xi:\Omega\rightarrow\mathcal{Y}$ be a stochastic process. Then for $y\in\mathcal{Y}$ 
\begin{equation}\label{eq:rv_def}
\langle y,\xi\rangle
\end{equation}
 defines a real-valued random variable. Assuming that 
 \begin{equation}
 \mathbb{E}(\langle \tilde{y},\xi\rangle^2)<\infty
 \end{equation}
for all $\tilde{y}\in\mathcal{Y}$ and that this expectation is continuous in $\tilde{y}$,
\[
\mathbb{E}(\langle \tilde{y},\xi\rangle\langle y,\xi\rangle)
\]
defines a continuous, symmetric nonlinear bilinearform. In particular, there exists the \textit{covariance operator} 
\[
\mathcal{C}:\mathcal{Y}\rightarrow\mathcal{Y}
\]
with
\[
\langle \mathcal{C} \tilde{y},y\rangle=\mathbb{E}(\langle \tilde{y},\xi\rangle\langle y,\xi\rangle).
\]

For the stochastic analysis of infinite dimensional problems via deterministic results, \eqref{eq:rv_def} is problematic. Namely, if $\{u_n\}_{n\in\mathbb{N}}$ is an orthonormal basis in $\mathcal{Y}$, the set $\{\langle u_n,\xi\rangle\}_{n\in\mathbb{N}}$ consists of infinitely many identically distributed random variables with $0<\mathbb{E}|\langle u_n,\xi\rangle|^2=\text{const}<\infty$ \cite{Louis}. Thus
\begin{equation}\label{eq:eerror}
\mathbb{E}\left(\sum_{n=1}^{\infty} |\langle u_n,\xi\rangle|^2\right)
\end{equation}
is almost surely infinite and a realization of the noise is an element of the Hilbert space $\mathcal{Y}$ with probability zero. Let us take the common example of Gaussian white noise which can be modeled via the above construction. Namely, with \[\mathbb{E}(\langle y,\xi\rangle)=0 \quad \forall y\in\mathcal{Y}\] and the covariance operator \[\mathcal{C}=\eta^2 I,\] where $I$ is the identity and $\eta>0$ the variance parameter, the Gaussian white noise is described \cite{Louis,Hanne14}. As consequence of (\ref{eq:eerror}) and explained for example in \cite{Hanne14}, a realization of such a Gaussian random variable is with probability zero an element of an infinite dimensional $L_2$-space. It is therefore inappropriate to use an $L_2$-norm for the residual in case of an infinite dimensional problem. Since in this case a realization of Gaussian white noise only lies (almost surely) in any Sobolev space $H^s$ with $s<-d/2$ where $d$ is the dimension of the domain, one should adjust the norm for the residual accordingly. Except for the paper \cite{Hanne14} this issue seems not to have been addressed in the literature. A main reason for this might be that for the practical solution of the Inverse Problem this is not a severe issue since in reality the measurements are finite dimensional and, in order to use a computer to solve the problem, a finite dimensional approximation of the unknown object has to be used. In this case the sum in \eqref{eq:eerror} is finite and the noise lies almost surely in the finite dimensional space. However, difficulties arise whenever one seeks to investigate convergence of the discretized problem to its underlying infinite dimensional problem. We will not address this issue and assume throughout the whole work that $\mathbb{E}||\epsilon||<\infty$ or use the slightly weaker bound on the Ky-Fan metric (see Section \ref{ssec:kyfan}). In order to handle the Ky Fan metric we need to be able to evaluate probabilities $\mathbb{P}(||y-y^\eta||>\varepsilon)$, $0\leq\varepsilon\leq1$, which is only meaningful if $y-y^\eta=:\epsilon\in \mathcal{Y}$. Assuming that $\mathcal{Y}$ is finite dimensional, then this is clear. For infinite dimensional problems, however, we have to assume that the noise is smooth enough for the sum in (\ref{eq:eerror}) to converge. Examples for this are Brownian noise ($1/f^2$-noise) or pink noise ($1/f$-noise), see e.g. \cite{gardner,kogan}. At this point we would also like to mention that as a consequence of our rather generic noise model we might not make use of some specific properties of the noise as would be possible when focusing on a particular  distribution of the noise. However, we are able to show convergence for a large variety of regularization methods.

\section{The Ky Fan metric}\label{ssec:kyfan}
In order to measure the magnitude of the stochastic noise and the quality of the reconstructions, we need metrics that incorporate the stochastic  nature of the problem. One such metric, which will be the the main tool for our stochastic convergence analysis, is the Ky Fan metric (cf. \cite{kyfan}). It is defined as follows.
\begin{definition}
Let $X_1$ and $X_2$ be random variables in a probability space $(\Omega,\mathcal{F},\mathbb{P})$ with values in a metric space $(\chi,d_\chi)$. The distance between $X_1$ and $X_2$ in the \textit{Ky Fan metric} is defined as
\begin{equation}\label{def:kyfan}
\rho_K(X_1,X_2):=\inf_{\varepsilon>0}\{\mathbb{P}(\{\omega\in\Omega:d_\chi(X_1(\omega),X_2(\omega))>\varepsilon\})<\varepsilon\}.
\end{equation}
\end{definition}
We will often drop the explicit reference to $\omega$. This metric essentially allows to lift results from a metric space to the space of random variables as the connection to the deterministic setting is inherent via the metric $d_\chi$ used in its definition. The deterministic metric is often induced by a norm $||\cdot||$. We will implicitly assume that equation \eqref{eq:problem} is scaled appropriately since $\rho_K(X_1,X_2)\leq1$ $\forall X_1,X_2$ by definition. Note that one can use definition \eqref{def:kyfan} also if $d_\mathcal{X}$ is a more general distance function than a metric. Then the construction \eqref{def:kyfan} itself is no longer a metric, however, the techniques used in later parts of the paper can readily be expanded to this setting.

 An immediate consequence of (\ref{def:kyfan}) is that $\rho_K(X_1,X_2)=0$ if and only if $X_1=X_2$ almost surely. Convergence in the Ky Fan metric is equivalent to convergence in probability, i.e., for a sequence $\{X_k\}_{k\in\mathbb{N}}\in\mathcal{X}$ and $X\in\mathcal{X}$ one has
\begin{equation*}
\rho_K(X_k,X)\overset{k\rightarrow\infty}{\longrightarrow} 0\quad\Leftrightarrow\quad \forall \varepsilon>0:\quad \mathbb{P}(d_\mathcal{X}(X_k,X)>\varepsilon)\overset{k\rightarrow\infty}{\longrightarrow}0.
\end{equation*}
Hence convergence in the Ky Fan metric also leads to pointwise (almost sure) convergence of certain subsequences in the metric $d_\chi$ \cite{DUD1989}.

A somewhat more intuitive and more frequently used stochastic metric is the expectation, or more general, a (stochastic) $L_p$ metric. For random variables $Y_1$ and $Y_2$ with values in a metric space $(\chi,d_\mathcal{Y})$,
\[
\mathbb{E}(d_\mathcal{Y}(Y_1,Y_2)^p)=\int_\Omega d_\mathcal{Y}(Y_1,Y_2)^p d\mathbb{P}(\omega)
\]
defines the $p$-th moment of $d_\mathcal{Y}(Y_1,Y_2)$ for $p\geq 1$, assuming the existence of the integral. We will use $p=1$ and refer to it as \textit{convergence in expectation}. Note that since the variance is defined as
\[
\text{Var}(d_\mathcal{Y}(Y_1,Y_2))=\mathbb{E}(d_\mathcal{Y}(Y_1,Y_2)^2)-E(d_\mathcal{Y}(Y_1,Y_2))^2\geq0
\]
one always has 
\begin{equation}\label{eq:expleqexpsq}
\mathbb{E}(d_\mathcal{Y}(Y_1,Y_2))\leq\sqrt{\mathbb{E}(d_\mathcal{Y}(Y_1,Y_2)^2)}.
\end{equation}
We will show later that for parameter choice rules the expectation of the noise has to be slightly overestimated, hence estimating $\mathbb{E}(d_\mathcal{Y}(y,y^\eta))$ via the popular and often easier to compute $L_2$-norm $E(d_\mathcal{Y}(y,y^\eta)^2)$ with \eqref{eq:expleqexpsq} is not problematic.

While the main part of our analysis is based on the description of the noise and the reconstruction quality in the Ky Fan metric, we will also allow the expectation as measure of the stochastic noise and partially show convergence of the reconstructed solutions in expectation. To this end, we comment in the following on the connection between those two metrics.

It is well-known that convergence in expectation implies convergence in probability, see for example \cite{DUD1989}. Hence, convergence in the Ky Fan metric is implied by convergence in expectation (and also by convergence of higher moments). Namely, with Markovs inequality one has, for an arbitrary nonnegative random variable $X$ with $E(X)<\infty$ and $C>0$
\begin{equation}\label{eq:markov}
\mathbb{P}(X>C)\leq \frac{\mathbb{E}(X)}{C}.
\end{equation}
Under an additional assumption one can show conversely that convergence in probability implies convergence in expectation. We have the following definition.
\begin{definition}[\cite{bogachev}, Definition A.3.1.]\label{def:uniform_int}
Let $(\Omega,\mathcal{F},\mathbb{P})$ be a complete probability space. A family $\mathcal{G}\subset L_1(\mathbb{P})$ is called \textit{uniformly integrable} if
\[
\lim_{C\rightarrow\infty}\sup_{x\in\mathcal{G}}\int_{|x|>C}|x(t)|\mathbb{P}(dt)=0
\]
\end{definition}
\begin{theorem}[\cite{bogachev}, Theorem A.3.2.]\label{thm:uniform_int}
Let $\{x_k\}_{k\in\mathbb{N}}\subset L_1(\mathbb{P})$ be a sequence convergent almost everywhere (or in probability) to a function $x$. If the sequence $\{x_k\}_{k\in\mathbb{N}}$ is uniformly integrable, then it converges to $x$ in the norm of $L_1(\mathbb{P})$.
\end{theorem}
From a practical point of view, uniform integrability of a sequence of regularized solutions to an Inverse Problem is a rather natural condition. Since Inverse Problems typically arise from some real-world application, it is to be expected that the true solution is bounded. For example, in Computer Tomography, the density of the tissue inside the body cannot be arbitrarily high. Although for an Inverse Problem with a stochastic noise model, boundedness of the regularized solutions can not be guaranteed due to the possibly huge measurement error, one can enforce the condition from a priori knowledge of the solution.

Assume that the true solution $x^\dag$ fulfills $||x^\dag||\leq\varrho$ and $|x^\dag|\leq C$ globally for some fixed $\varrho,C>0$. Under this assumption, let $\{x_k^{\eta(k)}\}_{k\in\mathbb{N}}$ be a sequence of regularized solution for noisy data with variance $\eta(k)\overset{k\rightarrow\infty}{\rightarrow}0$. Let $C_1,C_2>1$ and define
\begin{equation}\label{eq:sol_exp_conv_bound}
\tilde{x}_k^\eta:=\begin{cases} x_k^\eta, & ||x_k^\eta||\leq C_1\varrho,|x_k^\eta|\leq C_2C\\ 0, & \textup{otherwise} \end{cases}.
\end{equation}
Then the sequence $\{\tilde{x}_k^\eta\}_{k\in\mathbb{N}}$ is uniformly integrable. In other words, by discarding solutions that must be far away from the true solution in regard of a priori knowledge, convergence in the Ky Fan metric implies convergence in expectation.

To close this section, let us remark on the computation of the Ky Fan distance. It can be estimated via the moments of the noise.
\begin{theorem}
Let $Y_1,Y_2$ be random variables in a complete probability space $(\Omega,\mathcal{F},\mathbb{P})$ and $\mathbb{E}(d_\mathcal{Y}(Y_1,Y_2)^s)<\infty$ for some $s\in\mathbb{N}$. Then
\begin{equation}\label{eq:kyfan_bound_exp}
\rho_K(Y_1,Y_2)\leq \sqrt[s+1]{\mathbb{E}(d_\mathcal{Y}(Y_1,Y_2)^s)}
\end{equation}
\end{theorem}
\begin{proof}
One has, due to Markov's inequality \eqref{eq:markov} and the monotonicity of the mapping $z\mapsto z^s$ for $z\geq 0$,
\[
\mathbb{P}(d_\mathcal{Y}(Y_1,Y_2)>C)=\mathbb{P}(d_\mathcal{Y}(Y_1,Y_2)^s>C^s)\leq \frac{\mathbb{E}(d_\mathcal{Y}(Y_1,Y_2)^s)}{C^s}
\]
for $C\geq 0$. Solving $C=\frac{\mathbb{E}(d_\mathcal{Y}(Y_1,Y_2)^s)}{C^s}$ for $C$ yields the claim.
\end{proof}
Note that even if moments exist for all $s\in\mathbb{N}$
\[
\lim_{s\rightarrow\infty} \sqrt[s+1]{\mathbb{E}(d_\mathcal{Y}(Y_1,Y_2)^s)}\neq \mathbb{E}(d_\mathcal{Y}(Y_1,Y_2)),
\]
see \cite{HOF06,diss}, due to the tail of the distributions. In the Gaussian case, a direct estimate has been derived in \cite{NeuPik,HofPik07}. We present it in Proposition \ref{prop:kyfandistance}.

\section{Convergence in the stochastic setting}\label{ssec:conv_stoch}
\subsection{Deterministic Inverse Problems with stochastic noise}
As mentioned previously, the intention of this paper is to show convergence for Inverse Problems under a stochastic noise model using results from the deterministic setting. 
Assume we have at hand a deterministic regularization method of our liking for the solution of \eqref{eq:problem} under the noisy data \eqref{eq:ad_noise} where now $d_\mathcal{Y}(y,y^\delta)\leq \delta$ for some $\delta>0$. By regularization method we understand (possibly nonlinear) mappings
\begin{equation}
R_\alpha:\mathcal{D}(R_\alpha)\subset\mathcal{Y}\rightarrow\mathcal{X},\quad y^\delta\mapsto x_\alpha^\delta
\end{equation}
where $x_\alpha^\eta=R_\alpha(y^\eta)$ is the regularized solution to the regularization parameter $\alpha$ given the data $y^\eta$. Often, $x_\alpha^\delta$ is obtained via the minimization of functionals of the type \eqref{eq:functional}. In order to deserve the name regularization we require $R_\alpha$ to fulfill
\begin{equation}
\lim_{\delta\rightarrow0} ||R_\alpha(y^\delta)-x^\dag||=0
\end{equation}
under a certain choice of the regularization parameter $\alpha$ chosen either a priori $\alpha=\alpha(\delta)$ or a posteriori $\alpha=\alpha(\delta,y^\delta)$. In our notation  $x^\dag$ is the true solution, usually the minimum norm solution with respect to the penalty $\Phi$ in \eqref{eq:functional}, i.e., 
\[
\Phi(x^\dag)\leq\Phi(\bar{x})\quad\text{for all}\quad \bar{x}: F(\bar{x})=y.
\]
Note that, in particular for nonlinear problems, $x^\dag$ does not need to be unique. In \cite{HOF06,Hof_diss_book} it was pointed out that this is problematic for the lifting arguments. A standard argument in the deterministic theory is to prove convergence of subsequences to the desired solution, and then deduces convergence of the whole series of regularized solutions, if possible. In the stochastic setting, this is not possible in general since subsequences for different $\omega$ do not have to be related. A constructed example for this behavior can be found in Section 4.1. of \cite{HOF06,Hof_diss_book}. In order to lift general deterministic regularization methods into the stochastic setting we must therefore require that $x^\dag$ is unique. We formulate our convergence results assuming the noise to be bounded with respect to the Ky fan metric or in expectation. As we will see, in the latter case we have to ``inflate'' the expectation for decreasing variance $\eta$ in order to obtain convergence.
For the analysis we mainly use a lifting argument using deterministic theory. In \cite[Theorem 4.1]{HOF06,Hof_diss_book}, it was proved how by means of the Ky Fan metric deterministic results can be lifted to the space of random variables for nonlinear Tikhonov regularization. Since the theorem is based solely on the fact that there is a deterministic regularization theory and that the probability space $\Omega$ can be decomposed into a part where the deterministic theory holds and a small part where it does not, it is easily generalized. Before we state the Theorem, we need the following Lemmata.
\begin{lemma}(\cite{egoroff}, see also \cite{DUD1989})\label{lem:egoroff}
Let $(\Omega,\mathcal{F},\mathbb{P})$ be a complete probability space. Let $x_k$ and $x$ be measurable functions from $\Omega$ into a metric space $\chi$ with metric $d_\chi$. Suppose $x_k(\omega)\overset{d_\chi}{\rightarrow} x(\omega)$ for $\mathbb{P}$-almost all $\omega\in\Omega$.Then for any $\varepsilon>0$ there is a set $\tilde{\Omega}$ with $\mathbb{P}(\Omega\backslash \tilde{\Omega})<\varepsilon$ such that $x_k\overset{d_\chi}{\rightarrow} x(\omega)$ uniformly on $\tilde{\Omega}$, that is
\begin{equation*}
\lim_{k\rightarrow\infty}\sup \{d_\chi(x_k(\omega),x(\omega)):\omega\in \tilde{\Omega}\}=0.
\end{equation*}
\end{lemma}
\begin{lemma}[\cite{HOF06,Hof_diss_book}, Proposition 1.10]\label{prop:kyfan}
Let $\{x_k\}_{k\in\mathbb{N}}$ be a sequence of random variables that converges to $x$ in the Ky Fan metric. Then for any $\nu>0$ and $\varepsilon>0$ there exist $\tilde{\Omega}\subset\Omega$, $\mathbb{P}(\tilde{\Omega})\geq1-\varepsilon$, and a subsequence $x_{k_j}$ with
\begin{equation*}
d_\mathcal{X}(x_{k_j}(\omega),x(\omega))\leq(1+\nu)\rho_K(x_{k_j},x)\qquad\forall\omega\in\tilde{\Omega}.
\end{equation*}
Furthermore there exists a subsequence that converges to $x$ almost surely.
\end{lemma}
\begin{proof}We give a sketch of the proof for the first statement taken from \cite{HOF06,Hof_diss_book}.\\
Set $\sigma_k:=(1+\nu)\rho_K(x_k,x)$. By definition of the Ky Fan metric (\ref{def:kyfan}), for given $\sigma_k$, there exists a set $\Omega_{\sigma_k}$ with $\mathbb{P}(\Omega_{\sigma_k})\geq 1-\sigma_k$ and $\omega\in \Omega_{\sigma_k}$ such that $ d_\mathcal{X}(x(\omega),x_k(\omega))\leq \sigma_k$. For arbitrary $\varepsilon>0$ and $\sigma_k\rightarrow0$ we pick a subsequence $(\sigma_{k^j})$ with $\sum_{j=1}^\infty \sigma_{k^j} \leq \varepsilon$ and introduce the set $\tilde{\Omega}:=\bigcap_{j=1}^\infty \Omega_{\sigma_{k^j}}$. One can check that $\mathbb{P}(\tilde{\Omega})\geq 1-\varepsilon$. Since $\tilde{\Omega}$ is a subset of every $\Omega_{\sigma_{k^j}}$ we have
\begin{equation*}
\forall \omega\in \tilde{\Omega} \subseteq \Omega_{\sigma_{k^j}}: d_\mathcal{X}(x(\omega),x_{k^j}(\omega))\leq \sigma_{k^j},
\end{equation*}
which proves he first statement. The second one follows since convergence in Ky Fan metric is equivalent to convergence in probability, which itself implies almost-sure convergence of a subsequence, cf \cite{DUD1989}.
\end{proof}

With this, we are ready for the convergence theorem which we shall split in two parts, one for the Ky Fan metric as error measure and one for the expectation.
\begin{theorem}\label{thm:lifting_convergence_kyfan}
Let $R_\alpha$ be a regularization method for the solution of \eqref{eq:problem} in the deterministic setting under a suitable choice of the regularization parameter.
Let now $y^\eta=y+\epsilon(\eta)$ where $\epsilon(\eta)$ is a stochastic error such that $\rho_K(y,y^\eta)\rightarrow 0$ as $\eta\rightarrow 0$. Then, assuming \eqref{eq:problem} has a unique solution $x^\dag$ and all necessary assumptions for the deterministic theory (except the bound on the noise) hold with probability one, the regularization method $R_\alpha$ fulfills
\[
\lim_{\eta\rightarrow 0} \rho_K(x^\dag,R_\alpha(y^\eta))=0
\]
under the same parameter choice rule as in the deterministic setting with $\delta$ replaced by $\rho_K(y,y^\eta)$. 
If the regularized solutions are defined by \eqref{eq:sol_exp_conv_bound}, then it holds that
\[
\lim_{\eta\rightarrow 0}\mathbb{E}(d_\mathcal{X}(x^\dag,R_\alpha(y^\eta)))=0.
\]
\end{theorem}
\begin{proof}
Denote $x_\alpha(\eta):=R_\alpha(y^\eta)$. Define $\theta:=\limsup_{k\rightarrow \infty} \rho_K(x^\dag,x_{\alpha(\eta_k)})$. (Note that $0\leq\theta\leq1$ due to the properties of the Ky Fan metric). We show in the following that for arbitrary $\varepsilon>0$ we have $\theta/2\leq\varepsilon$ and hence 
\[
\limsup_{k\rightarrow \infty}\rho_K(x^\dag,x_{\alpha(\eta_k)})=\lim_{k\rightarrow \infty}\rho_K(x^\dag,x_{\alpha(\eta_k)})=0.
\]
As a first step we pick a ``worst case" subsequence $\{y^ {\eta_{k^j}}\}$ of $\{y^{\eta_k}\}$, a subsequence for which the corresponding solutions satisfy $\rho_K(x^\dag,x_{\alpha(\eta_{k^j})})\geq\theta/2$. We now show that even from this ``worst case" sequence we can pick a subsequence $\{y^{\eta_{k^j_l}}\}$ for which we have $\limsup \rho_K(x^\dag,x_{\alpha(\eta_{k^j_l})})\leq \varepsilon$ for arbitrary $\varepsilon>0$.\\
Let $\varepsilon>0$. According to Lemma \ref{prop:kyfan} we can pick a subsequence $\{y^{\eta_{k^j_l}}\}$ and a set $\tilde{\Omega}$ with $\mathbb{P}(\tilde{\Omega})\geq 1- \frac{\varepsilon}{2}$ as well as $d_\mathcal{Y}(y(\omega),y^{\eta_{k^j_l}}(\omega))\leq (1+\nu)\rho_K(y,y^{\eta_{k^j_l}})$, $\nu>0$ arbitrarily small, on $\tilde{\Omega}$. For all $\omega\in\tilde{\Omega}$, the noise tends to zero. We can therefore use the deterministic result with $\delta=\rho_K(y(\omega),y^{\eta_{k^j_l}})$ and deduce that $x_{\alpha(\eta_{k^j_l})}(\omega)$ converges to the unique solution $x^\dag(\omega)$ for $\eta_{k^j_l}\rightarrow0$, $\omega\in\tilde{\Omega}$ where in the choice of the regularization parameter $\delta$ is substituted by $\rho_K(y(\omega),y^{\eta_{k^j_l}})$. The convergence is not uniform in $\omega$; nevertheless, pointwise convergence implies uniform convergence except on sets of small measure according to Lemma \ref{lem:egoroff}. Therefore there exist $\tilde{\Omega}^\prime\subset\tilde{\Omega}$, $\mathbb{P}(\tilde{\Omega}^\prime)<\frac{\varepsilon}{2}$ and $j_0\in\mathbb{N}$ such that $d_\mathcal{X}(x_{\alpha(\eta_{k^j_l})}(\omega),x^\dag(\omega))<\varepsilon$ $\forall \omega\in\tilde{\Omega}\backslash\tilde{\Omega}^\prime$ and $j\geq j_0$. We thus have
\begin{equation*}
\mathbb{P}\left(\left\lbrace \omega\in\tilde{\Omega}: d_\mathcal{X}(x_{\alpha(\eta_{k^j_l})}(\omega),x^\dag(\omega)) > \varepsilon\right\rbrace\right)\leq\mathbb{P}(\tilde{\Omega}^\prime)\leq\varepsilon/2.
\end{equation*}
Since we split $\Omega=\Omega\backslash\tilde{\Omega}\cup\tilde{\Omega}\backslash\Omega^\prime_\varepsilon\cup\Omega^\prime_\varepsilon$ with $\mathbb{P}(\Omega\backslash\tilde{\Omega})<\frac{\varepsilon}{2}$, $\mathbb{P}(\Omega\backslash\tilde{\Omega})+\mathbb{P}(\tilde{\Omega}^\prime)\leq\varepsilon$ we have shown existence of a subsequence $\eta_{k^j_l}$ such that
\begin{equation*}
\mathbb{P}\left(\left\lbrace\omega\in\Omega:d_\mathcal{X}(x_{\alpha(\eta_{k^j_l})}(\omega),x^\dag(\omega))>\varepsilon\right\rbrace\right)\leq\varepsilon
\end{equation*}
for $\eta_{k^j_l}$ sufficiently small. This $\varepsilon$ is by definition of the Ky Fan metric an upper bound for the distance between $x_{\alpha(\eta_{k^j_l})}$ and $x^\dag$. Therefore we have
\[
\limsup_{l\rightarrow\infty}\rho_K(x_{\alpha(\eta_{k^j_l})},x^\dag)\leq \varepsilon.
\]
On the other hand, the original sequence satisfied $\liminf_{j\rightarrow\infty}\rho_K(x^\dag,x_{\alpha(\eta_{k^j})})\geq \theta/2$. Since $\liminf_{j\rightarrow\infty}\rho_K(x^\dag,x_{\alpha(\eta_k)})\leq\limsup_{l\rightarrow\infty}\rho_K(x_{\alpha(\eta_{k^j_l})},x^\dag)$ it follows $\theta/2\leq\varepsilon$. Because $\varepsilon>0$ was arbitrary, this implies $\theta=0$, which concludes the proof of convergence in the Ky Fan metric. Convergence in expectation follows from Theorem \ref{thm:uniform_int} noting that by \eqref{eq:sol_exp_conv_bound} the sequence of regularized solutions is uniformly integrable.
\end{proof}
\begin{theorem}\label{thm:lifting_convergence_exp}
Let $R_\alpha$ be a regularization method for the solution of \eqref{eq:problem} in the deterministic setting under a suitable choice of the regularization parameter.
Let now $y^\eta=y+\epsilon(\eta)$ where $\epsilon(\eta)$ is a stochastic error such that $\mathbb{E}(d_\mathcal{Y}(y,y^\eta))\rightarrow 0$ as $\eta\rightarrow 0$. Then, assuming \eqref{eq:problem} has a unique solution $x^\dag$ and all necessary assumptions for the deterministic theory (except the bound on the noise) hold with probability one, the regularization method $R_\alpha$ fulfills
\[
\lim_{\eta\rightarrow 0} \rho_K(x^\dag,R_\alpha(y^\eta))=0
\]
under the same parameter choice rule as in the deterministic setting with $\delta$ replaced by $\tau(\eta)\mathbb{E}(d_\mathcal{Y}(y,y^\eta))$ where $\tau(\eta)$ fulfills
\begin{equation}\label{eq:tau_reqs}
\tau(\eta)\overset{\eta\rightarrow0}{\rightarrow}\infty \quad\text{and}\quad \lim_{\eta\rightarrow0}\tau(\eta)\mathbb{E}(d_\mathcal{Y}(y,y^\eta))=0.
\end{equation} 
If the regularized solutions are defined by \eqref{eq:sol_exp_conv_bound} then it holds that
\[
\lim_{\eta\rightarrow 0}\mathbb{E}(d_\mathcal{X}(x^\dag,R_\alpha(y^\eta)))=0.
\]
\end{theorem}
\begin{proof}
As previously we pick a ``worst case" subsequence $\{y^ {\eta_{k^j}}\}$ of $\{y^{\eta_k}\}$, a subsequence for which the corresponding solutions satisfy $\rho_K(x^\dag,x_{\alpha(\eta_{k^j})})\geq\theta/2$.\\
Let $\varepsilon>0$. We can now pick a subsequence which we again denote by $\{y^{\eta_{k^j_l}}\}$ fulfilling $\frac{2}{\tau(\eta_{k^j_l})}\leq\varepsilon$, where without loss of generality $t(\eta_{k^j_l})>1$, such that 
\[
\mathbb{P}(\omega:d_\mathcal{Y}(y(\omega)-y^{\eta_{k^j_l}}(\omega))>\tau({\eta_{k^j_l}})\mathbb{E}(d_\mathcal{Y}(y,y^{\eta_{k^j_l}})))\leq \frac{1}{\tau(\eta_{k^j_l})}\leq\frac{\varepsilon}{2}.
\]
This again defines, via the complement in $\Omega$, $\tilde{\Omega}$ with $\mathbb{P}(\tilde{\Omega})\geq 1- \frac{\varepsilon}{2}$ on which $d_\mathcal{Y}(y(\omega),y^{\eta_{k^j_l}}(\omega))\leq \tau({\eta_{k^j_l}})\mathbb{E}(d_\mathcal{Y}(y,y^{\eta_{k^j_l}}))$. As before, we can now apply the deterministic theory by substituting $\delta$ with $\tau({\eta_{k^j_l}})\mathbb{E}(d_\mathcal{Y}(y,y^{\eta_{k^j_l}}))$. The remainder of the proof is identical to the one of Theorem \ref{thm:lifting_convergence_exp}.
\end{proof}

The theorems justify the use of deterministic algorithms under a stochastic noise model. Since the proof is solely based on relating the stochastic noise to a deterministic one on subsets of $\Omega$ and does not use any specific properties of the regularization methods or the underlying spaces, it opens most of the deterministic methods for the a stochastic noise model. In particular, the parameter choice rules from the deterministic setting are easily adapted.

As usual in deterministic literature, the general convergence theorem is followed by convergence rates which are obtained under additional assumptions. Often these conditions ensure at least local uniqueness of the true solution. If not, we have to require such a property for the same reason as previously. 
\begin{theorem}\label{thm:lifting_rates_kyfan}
Let $R_\alpha$ be a regularization method for the solution of \eqref{eq:problem} in the deterministic setting such that, under a set of assumptions on the operator $F$ and the solutions $x^\dag$ and a suitable choice of the regularization parameter,
\[
d_\mathcal{X}(x^\dag,R_\alpha(y^\delta))\leq \varphi(d_\mathcal{Y}(y,y^\delta))
\]
with a monotonically increasing right-continuous function $\varphi$.

Let now $y^\eta=y+\epsilon(\eta)$ where $\epsilon(\eta)$ is a stochastic error such that 
\begin{itemize}\item[a)] $\rho_K(y,y^\eta)\rightarrow 0$ or
\item[b)] $\mathbb{E}(d_\mathcal{Y}(y,y^\eta))\rightarrow 0$ 
\end{itemize}
as $\eta\rightarrow 0$. Then, assuming all necessary assumptions for the deterministic theory (except the bound on the noise) hold with probability one and that there is (either by the deterministic conditions or by additional assumption) a (locally) unique solution $x^\dag$ to \eqref{eq:problem}, the regularization method $R_\alpha$ fulfills
\[
\rho_K(x^\dag,R_\alpha(y^\eta))=\mathcal{O}(\max\{\varphi(\rho_K(y,y^\eta)),\rho_K(y,y^\eta)\})
\]
in case a) or, respectively, in case b),
\[
\rho_K(x^\dag,R_\alpha(y^\eta))=\mathcal{O}(\max\{\varphi(\tau(\eta)\mathbb{E}(d_\mathcal{Y}(y,y^\eta))),\mathbb{P}(d_\mathcal{Y}(y,y^\eta)\geq \tau(\eta)\mathbb{E}(d_\mathcal{Y}(y,y^\eta)))\})
\]
under the same parameter choice rule as in the deterministic setting with $\delta$ replaced by $\rho_K(y,y^\eta)$ (case a)) or $\tau(\eta)\mathbb{E}(d_\mathcal{Y}(y,y^\eta))$ where $\tau(\eta)$ fulfills \eqref{eq:tau_reqs} (case b)).
\end{theorem}
\begin{proof}
We start again with the Ky Fan distance as noise measure. Since we have the deterministic theory at hand, we know that $d_\mathcal{X}(x^\dag,R_\alpha(y^\eta))\leq C \varphi(\delta)$ whenever $d_\mathcal{Y}(y,y^\eta)\leq\delta$. With $\delta=\rho_K(y,y^\eta)$ we have, since $\varphi$ is monotonically increasing and right continuous,
\begin{align*}
\mathbb{P}(d_\mathcal{X}(x^\dag,R_\alpha(y^\eta))&>\varphi(\rho_K(y,y^\eta)))\leq \mathbb{P}(d_\mathcal{Y}(y,y^\eta)\geq\rho_K(y,y^\eta))\\
&\leq1-\mathbb{P}(d_\mathcal{Y}(y,y^\eta)\leq\rho_K(y,y^\eta))\\
&\leq1-(1-\rho_K(y,y^\eta))=\rho_K(y,y^\eta)
\end{align*}
and hence by definition $\rho_K(x^\dag,R_\alpha(y^\eta))=\mathcal{O}(\max\{\varphi(\rho_K(y,y^\eta)),\rho_K(y,y^\eta)\})$.

If the expectation is used as measure for the data error, we have
\[
\mathbb{P}(d_\mathcal{Y}(y,y^\eta)\geq\tau(\eta)\mathbb{E}(d_\mathcal{Y}(y,y^\eta)))\leq\frac{1}{\tau(\eta)}
\]
by Markovs inequality. Hence, with probability $1-\frac{1}{\tau( \eta)}$ we are in the deterministic setting with $\delta=\tau(\eta)\mathbb{E}(d_\mathcal{Y}(y,y^\eta))$ and
\begin{align*}
\mathbb{P}(d_\mathcal{X}(x^\dag,R_\alpha(y^\delta))&> \varphi(\tau(\eta)\mathbb{E}(d_\mathcal{Y}(y,y^\eta))))\\&\leq \mathbb{P}(d_\mathcal{Y}(y,y^\eta)>\tau(\eta)\mathbb{E}(d_\mathcal{Y}(y,y^\eta)))\leq\frac{1}{\tau( \eta)}.
\end{align*}
The convergence rate follows by the definition of the Ky Fan metric.
\end{proof}
For Inverse Problems, the convergence rates are most often given by functions which decay at most linearly fast, i.e., \[
\max\{\varphi(\rho_K(y,y^\eta)),\rho_K(y,y^\eta)\}=\varphi(\rho_K(y,y^\eta)).
\]
Hence in this case the convergence rates are preserved in the Ky Fan metric. For the expectation this is not the case. We have to gradually inflate the expectation by the parameter $\tau$ in order to obtain convergence (and rates). Let us discuss the simple example of Gaussian noise in the finite dimensional setting, i.e. $\epsilon$ from \eqref{eq:ad_noise} consists of $m\in\mathbb{N}$ i.i.d. random variables $\epsilon_i\sim\mathcal{N}(0,\eta^2I_m)$ with zero mean and variance $\eta^2$. Then it has been shown in \cite{diss} that for any $\tau>1$ 
\begin{equation}\label{eq:exp_schaetzung}
\mathbb{P}(||\eps||_2\geq\tau\mathbb{E}(||\eps||_2))=\frac{\Gamma(\frac{m}{2},(\tau\Gamma(\frac{m+1}{2})/\Gamma(\frac{m}{2}))^2)}{\Gamma(\frac{m}{2})}
\end{equation} 
with the gamma functions $\Gamma(\cdot)$ and $\Gamma(\cdot,\cdot)$ defined as
\begin{equation*}
\Gamma(a)=\int_0^\infty t^{a-1}e^{-t}dt,\quad\Gamma(a,z)=\int_z^\infty t^{a-1}e^{-t}dt.
\end{equation*}
In particular, \eqref{eq:exp_schaetzung} is independent of the variance $\eta^2$. In order to to decrease the probability to zero, we therefore have to link $\tau$ with the variance. For Gaussian noise of the above kind the following estimate for the Ky Fan distance between true and noisy data has been given in \cite{NeuPik}.
\begin{proposition}\label{prop:kyfandistance}
Let $\xi$ be a random variable with values in $\mathbb{R}^m$. Assume that the distribution of $\epsilon$ is $\mathcal{N}(0,\eta^2I_m)$ with $\sigma>0$. Then it holds in $(\mathbb{R}^m,||\cdot||_2)$ that
\begin{equation}\label{def:error_neupik}
\rho_K(\epsilon,0)\leq \min\left\lbrace 1,\sqrt{2}\eta\sqrt{m-\min\left\lbrace\ln\left(\eta^22\pi m^2\left(\frac{e}{2}\right)^m\right),0\right\rbrace} \right\rbrace.
\end{equation}
\end{proposition}
Recall that
\begin{equation}\label{eq:expe_gauss}
\mathbb{E}(||\epsilon||_2)=\eta\Gamma\left(\frac{1+m}{2}\right)/\left(\sqrt{2}\Gamma\left(\frac{m}{2}\right)\right)\leq \eta\sqrt{m},
\end{equation}
see e.g. \cite{diss}. Comparing \eqref{def:error_neupik} and \eqref{eq:expe_gauss}, one sees that $\mathbb{E}(||\epsilon||_2)<\rho_K(\epsilon,0)$ and in particular the decay of $\rho_K(\epsilon,0)$ slows down with decreasing $\eta$. In other words, the artificial inflation we had to impose on the expectation is automatically included in the Ky Fan distance which we suppose is the reason why the convergence theory carries over in such a direct fashion for the Ky Fan metric.

For many nonlinear Inverse Problems the requirement of a unique solution is too strong. Often one has several solutions of the same quality, in particular there exists more than one minimum norm solution. In this case, Theorem \ref{thm:lifting_convergence_kyfan} is not applicable. In the example \cite[Example 4.3 and 4.5]{HOF06,Hof_diss_book} with two minimum norm solutions the noise was constructed such that, while the error in the data converges to zero, for each fixed $\omega\in\Omega$ the regularized solutions jump between both solutions such that no converging subsequence can be found. The main problem there is that the Ky Fan distance cannot incorporate the concept that all minimum norm solutions are equally acceptable. We will now define a pseudo metric that resolves this issue.
\begin{definition}
Let $(\mathcal{X},d_{\mathcal{X}})$ be a metric space. Denote with $\mathcal{L}$ the set of minimum-norm solutions to \eqref{eq:problem}. 
Then
\begin{equation}
\rho_K^\mathcal{L}(x):=\inf_{\varepsilon>0}\left\lbrace \mathbb{P}\left( \inf_{x^\dag\in\mathcal{L}} d_{\mathcal{X}}(x,x^\dag)>\varepsilon\right)\leq\varepsilon\right\rbrace
\end{equation}
measures the distance between an element $x\in\mathcal{X}$ and the set $\mathcal{L}$, in particular it is
\[
\rho_K^\mathcal{L}(x)=0\quad\Leftrightarrow \quad x\in\mathcal{L} \quad\text{almost surely.}
\]
\end{definition}
With this, one can define a pseudometric on $(\Omega,\mathcal{F},\mathbb{P})$ via
\begin{equation}\label{def:pseudometric}
\rho_K^\mathcal{L}(x_1,x_2)=:\max\{\rho_K^\mathcal{L}(x_1),\rho_K^\mathcal{L}(x_2)\}.
\end{equation}
Obviously \eqref{def:pseudometric} is positive, symmetric and fulfills the triangle inequality. However, $\rho_K^\mathcal{L}(x_1,x_2)=0$ does not imply $x_1=x_2$ a.e. but instead $x_1\wedge x_2\in\mathcal{L}$ which fixes the aforementioned issue of the Ky Fan metric and allows the following theorems.
\begin{theorem}
Let $R_\alpha$ be a regularization method for the solution of \eqref{eq:problem} in the deterministic setting under a suitable choice of the regularization parameter.
Let now $y^\eta=y+\epsilon(\eta)$ where $\epsilon(\eta)$ is a stochastic error such that 
\begin{itemize}\item[a)] $\rho_K(y,y^\eta)\rightarrow 0$ or
\item[b)] $\mathbb{E}(d_\mathcal{Y}(y,y^\eta))\rightarrow 0$ 
\end{itemize}
as $\eta\rightarrow 0$. Then, assuming all necessary assumptions for the deterministic theory (except the bound on the noise) hold with probability one, the regularization method $R_\alpha$ fulfills
\[
\lim_{\eta\rightarrow 0} \rho_K^\mathcal{L}(R_\alpha(y^\eta))=0
\]
under the same parameter choice rule as in the deterministic setting with $\delta$ replaced by $\rho_K(y,y^\eta)$ (case a)) or $\tau(\eta)\mathbb{E}(d_\mathcal{Y}(y,y^\eta))$ where $\tau(\eta)$ fulfills \eqref{eq:tau_reqs} (case b)). In particular, the series of regularized solutions fulfills
\[
\lim_{\eta_1,\eta_2\rightarrow0}\rho_K^\mathcal{L}(R_\alpha(y^{\eta_1}),R_\alpha(y^{\eta_2}))=0
\]
\end{theorem}
\begin{proof}
The proof follows the lines of the one of Theorem \ref{thm:lifting_convergence_kyfan} with $\rho_K(\cdot,x^\dag)$ replaced by $\rho_K^\mathcal{L}(\cdot)$. Also Lemma \ref{lem:egoroff} is easily adjusted to incorporate multiple solutions.
\end{proof}

So far we assumed that only the noise is stochastic whereas the operator $F$ and the unknown $x$ were assumed to be deterministic. In \cite{HOF06,Hof_diss_book} general stochastic  Inverse Problems
\[
F(x(\omega),\omega)=y(\omega)
\]
were considered. It was shown how deterministic conditions such as source conditions can be incorporated into the stochastic setting by assuming that the deterministic conditions hold with a certain probability. However, additional conditions may occur when lifting these in order to ensure the deterministic requirements up to a certain probability. Since this is easier seen given an example, we move the discussion of the complete stochastic formulation in the next section. Although we will address only one particular example, the technique can be applied to general approaches.

\subsection{Fully stochastic Inverse Problems}\label{ssec:ex_hofinger}
Due to the possible multiplicity of stochastic conditions which might appear in this context it seems not possible to develop a lifting strategy in such a general fashion as in the previous section. We will therefore consider two classical examples, namely nonlinear Tikhonov regularization and Landweber's method for nonlinear Inverse Problems. The theory is taken completely from \cite{HOF06,Hof_diss_book}.

\subsubsection{Nonlinear Tikhonov Regularization} 
We seek the solution of a nonlinear ill-posed problem \eqref{eq:problem} via the variational problem
\[
x_\alpha^\delta=\text{argmin} ||F(x)-y^\delta||^2+\alpha||x-x^\ast||^2
\]
with a reference point $x^\ast\in X$ and given noisy data $y^\eta$ according to \eqref{eq:ad_noise} where the stochastic distribution of the noise is assumed to be known. We shall skip the general convergence theorem (which follows as in the previous section) and move to convergence rates directly. In the deterministic theory, i.e. when $y^\delta$ is the noisy data with $||y-y^\delta||\leq\delta$, we have the following theorem from \cite{EHN1996}.
\begin{theorem}\label{thm:tikh_nonlin_det}
Let $\mathcal{D}(\mathcal{F})$ be convex, $y^\delta\in\mathcal{Y}$ such that $||y-y^\delta||\leq\delta$ and $x^\dag$ denote the $x^\ast$-minimum norm solution of \eqref{eq:problem}. Furthermore let the following conditions hold.
\begin{itemize}
\item[a)] $F$ is Fr\'{e}chet-differentiable
\item[b)]There exists $\gamma\geq0$ such that $||F^\prime(x^\dag)-F^\prime(x)||\leq\gamma||x^\dag-x||$ in a sufficiently large ball $\mathcal{B}_\theta(x^\dag)\cap\mathcal{D}(F)$
\item[c)]$x^\dag-x^\ast$ satisfies the source condition $x^\dag-x^\ast=F^\prime(x^\dag)^\ast v$ for some $v\in\mathcal{Y}$.
\item[d)] The source element satisfies $\gamma||v||<1$.
\end{itemize}
Then for the choice $\alpha=c\delta$ with some fixed $c>0$ we obtain
\begin{equation}
||x^\dag-x_\alpha^\delta||\leq \frac{\delta+\alpha||v||}{\sqrt{\alpha}\sqrt{1-\gamma||v||}}=\mathcal{O}(\sqrt{\delta})\textup{ and }||F(x_\alpha^\delta)-y^\delta)||=\mathcal{O}(\delta).
\end{equation}
\end{theorem}
As given in Theorem 4.6 of \cite{HOF06}, the following stochastic formulation of Theorem \ref{thm:tikh_nonlin_det} holds. 
\begin{theorem}\label{thm:tikh_nonlin_stoch}
Let $\mathcal{D}(\mathcal{F})$ be convex, let $y^\eta$ be such that $0\leq\rho_K(y,y^\eta)<\infty$ and $x^\dag$ denote the $x^\ast$-minimum norm solution of \eqref{eq:problem} for almost all $\omega$. Furthermore let the following conditions hold.
\begin{itemize}
\item[a)] $F(.,\omega)$ is Frechet-differentiable for almost all $\omega$
\item[b)] $F^\prime(\cdot,\omega)$ satisfies
\[
||F^\prime(x^\dag(\omega),\omega)-F^\prime(x,\omega)||\leq\gamma(\omega)||x^\dag(\omega)-x||
\] in a sufficiently large ball $\mathcal{B}_\theta(x^\dag(\omega))\cap\mathcal{D}(F)$
\item[c)] (smoothness) $\mathbb{P}(\Omega_{sc})=1$ where 
\[
\Omega_{sc}:=\{\omega:\exists v(\omega), x^\dag(\omega)-x^\ast(\omega)=F^\prime(x^\dag(\omega),\omega)^\ast v(\omega)\}.
\]
\item[d)] (closedness) $\mathbb{P}(\omega\in\Omega{sc}:\gamma(\omega) ||v(\omega)||>\xi)<\phi_{cl}(\xi),\quad\lim_{\xi\rightarrow1^-}\phi_{cl}(\xi)=0$
\item[e)] (decay) $\mathbb{P}(\omega\in\Omega{sc}: ||v(\omega)||>\tau)<\varphi_{de}(\tau),\quad\lim_{\tau\rightarrow\infty}\varphi_{de}(\tau)=0$.
\end{itemize}
Then for the choice $\alpha\sim\rho_K(y,y^\eta)$ we obtain
\begin{equation}\label{eq:conv_rate_tikh_stoch}
\rho_K(x^\dag,x_\alpha^\eta)\leq\inf_{\substack{\tau<\infty\\\xi\in(0,1)}}\max\left\lbrace \rho_K(y,y^\eta)+\varphi_{cl}(\xi)+\varphi_{de}(\tau),\sqrt{\rho_K(y,y^\eta)}\frac{\mathcal{O}(1+\tau)}{\sqrt{1-\xi}} \right\rbrace.
\end{equation}
\end{theorem}
\begin{proof}
We have $||y-y^\eta||\leq\rho_K(y,y^\eta)$ with probability $1-\rho_k(y,y^\eta)$. Now fix $\xi<1$ and $0<\tau<\infty$. Then with probability $1-(\varphi_{cl}(\xi)+\varphi_{de}(\tau))$ conditions d) and e) are fulfilled. Thus for the corresponding values of $\omega$ we can apply Theorem \ref{thm:tikh_nonlin_det} and obtain
\[
||x^\dag(\omega)-x_\alpha^\eta(\omega)||\leq \frac{\rho_K(y,y^\eta)+\alpha\tau}{\sqrt{\alpha}\sqrt{1-\xi}}
\]
or, fixing the parameter $\alpha\sim\rho_K(y,y^\eta)$, 
\[
||x^\dag(\omega)-x_\alpha^\eta(\omega)||\leq \sqrt{\rho_K(y,y^\eta)}\frac{\mathcal{O}(1+\tau)}{\sqrt{1-\xi}}.
\]
This estimate holds on a set with probability greater or equal $1-(\rho_K(y,y^\eta)+\varphi_{cl}(\xi)+\varphi_{de}(\tau))$. The Ky Fan distance can therefore be bounded as
\[
\rho_K(x^\dag,x_\alpha^\eta)\leq\max\left\lbrace \rho_K(y,y^\eta)+\varphi_{cl}(\xi)+\varphi_{de}(\tau),\sqrt{\rho_K(y,y^\eta)}\frac{\mathcal{O}(1+\tau)}{\sqrt{1-\xi}} \right\rbrace.
\]
This estimate is valid for arbitrary choices of $\xi$ and $\tau$ above, therefore we may bound the Ky fan distance of $x^\dag$ and $x_\alpha^\eta$ by taking the infimum with respect to $\xi$ and $\tau$.
\end{proof}
The core principle of the lifting strategy is to ensure that there exists a subset $\tilde{\Omega}\subset(\Omega)$ such that all deterministic assumptions hold with probability one on $\tilde{\Omega}$. This may lead to the introduction of new conditions such as the decay condition in Theorem \ref{thm:tikh_nonlin_stoch}. Namely, since $\gamma(\omega)$ and $||v(\omega)||$ may vary with $\omega$, it may be possible that for a sequence $\{\omega_k\}_{k\in\mathbb{N}}$ $\gamma(\omega_k)\rightarrow0$ and $||v(\omega_k)||\rightarrow\infty$ such that still for all $k\in\mathbb{N}$ $\gamma(\omega(k))||v(\omega_k)||<1$. In this case the parameter $\tau$ cannot be treated as a constant in the convergence rate, but it influences it to a significant degree. The decay condition had to be imposed in order to control the growth of $\tau$. It is, however, possible to avoid condition e) by imposing other conditions. For example, one could require that $\gamma(\omega)$ is bounded below by some $0<c<1$. In this case condition d) implies e). A more detailed discussion is given in \cite{HOF06}.

Accordingly, in order to lift other deterministic convergence rate results into the fully stochastic setting, a careful examination of the conditions necessary for convergence in the stochastic setting, understanding their cross-connections and dependencies is important. However, once the conditions have been translated to the stochastic setting, convergence rates follow immediately using the Ky Fan metric. We will close this example by showing how particular choices of the stochastic parameters in Theorem \ref{thm:tikh_nonlin_stoch} influence the convergence rate. To this end, we cite Remark 4.8 of \cite{HOF06}.

Let in the first examples the operator be deterministic, i.e., $F(\cdot,\omega)=F(\cdot)$ where $\gamma(\omega)=\gamma=1$.

First consider the case that $||v||\in U[0,1]$, i.e., it is uniformly distributed on the interval $[0,1]$. We therefore have $\varphi_{cl}(\xi)=1-\xi$, as well as $\varphi_{de}=0$ for $\tau>1$. Thus Theorem \ref{thm:tikh_nonlin_stoch} implies
\[
\rho_K(x^\dag,x_\alpha^\eta)\leq\inf_{0<\alpha<\infty}\inf_{\xi\in(0,1)}\max\left\lbrace \rho_K(y,y^\eta)+1-\xi,\sqrt{\rho_K(y,y^\eta)}\frac{\rho_K(y,y^\eta)+\alpha}{\alpha\sqrt{1-\xi}} \right\rbrace
\]
which gives for $\alpha\sim\rho_K(y,y^\eta)$ the optimal rate
\[
\rho_K(x^\dag,x_\alpha^\eta)=\mathcal{O}(\rho_K(y,y^\eta)^{1/3}).
\]
For the second case suppose that $\varphi_{de}(\tau)=c\tau^{-e}$ for some exponent $e>0$. Since now we do not have $\varphi_{cl}(\xi)\rightarrow0$, but $\varphi_{cl}\geq c>0$ we obtain
\[
\rho_K(x^\dag,x_\alpha^\eta)\leq\inf_{0<\alpha<\infty}\inf_{\substack{t<\infty\\\xi\in(0,1)}}\max\left\lbrace c+c\tau^{-e},\sqrt{\rho_K(y,y^\eta)}\frac{\rho_K(y,y^\eta)+\alpha}{\alpha\sqrt{1-\xi}} \right\rbrace.
\]
Since the right hand side does not converge to zero we do not obtain a convergence rate anymore. However, convergence itself still follows from Theorem \ref{thm:lifting_convergence_kyfan}.

Finally, consider the case when both d) and e) from Theorem \ref{thm:tikh_nonlin_stoch} influence the convergence behavior, because $F$ is stochastic with varying $\gamma(\omega)$. For instance in the case the for some $\omega\in U[0,1]$ we have $x^\dag(\omega)=\omega x^\dag$ and $\gamma(\omega)=1-\omega$, we find that $\varphi_{cl}(\xi)=1-\xi$ and $\varphi_{de}(\tau)=c/(1+\tau)$ are compatible realizations of $\varphi_{cl}(\cdot)$ and $\varphi_{de}(\cdot)$. With this one can show 
\[
\rho_K(x^\dag,x_\alpha^\eta)=\mathcal{O}(\rho_K(y,y^\eta)^{1/4})
\]
under the parameter choice $\alpha\sim\rho_K(y,y^\eta)^{5/4}$. From the given examples it is evident that the convergence speed is heavily influenced by the conditions d) and e) in Theorem \ref{thm:tikh_nonlin_stoch}. Therefore, although the general formula for the convergence rate \eqref{eq:conv_rate_tikh_stoch} may suggest that the convergence rate is close to the deterministic one, it may be significantly slower due to the additional stochastic properties.

\subsubsection{Nonlinear Landweber iteration}
As before we seek the solution of a nonlinear ill-posed problem \eqref{eq:problem} given noisy data $y^\eta$ according to \eqref{eq:ad_noise} where the stochastic distribution of the noise is assumed to be known. Landweber's method can be seen as a descent algorithm for $||F(x)-y^\delta||^2$ and is defined via the iteration
\begin{equation}\label{eq:lw_det}
x_{k+1}^\delta=x_{k}^\delta-\gamma F^\prime(x_k^\delta)(F(x_k^\delta)-y^\delta),\quad k=1,2,\dots,
\end{equation}
where $\gamma>0$ is an appropriately chosen stepsize and $x_0^\delta$ an initial guess. Landweber's method constitutes a regularization method if it is stopped early enough \cite{EHN1996}. In the deterministic theory, i.e. when $y^\delta$ is the noisy data with $||y-y^\delta||\leq\delta$, we have the following theorem from \cite{EHN1996} for convergence rates of the Landweber method.
\begin{theorem}\label{thm:lw_nonlin_det}
Let $\mathcal{D}(\mathcal{F})$ be convex, $y^\delta\in\mathcal{Y}$ such that $||y-y^\delta||\leq\delta$ and $x^\dag$ denote the $x^\ast$-minimum norm solution of \eqref{eq:problem}. Assume \eqref{eq:problem} has a solution in $\mathcal{B}_\vartheta(x^\ast)$. Furthermore let the following conditions hold on $\mathcal{B}_{2\vartheta}(x^\ast)$.
\begin{itemize}
\item[a)] $F$ is Frechet-differentiable with $||F^\prime(x)||\leq 1$ and
\[
||F(x)-F(x^\dag)-F^\prime(x^\dag)(x-x^\dag)||\leq\zeta||F(x)-F(x^\dag)||,\quad 0<\zeta<\frac{1}{2}
\]
\item[b)]$F^\prime(x)=R_xF^\prime(x^\dag)$  where the bounded linear operators $R$ satisfy $||R_x-I||\leq C||x-x^\dag||$
\item[c)]$x^\dag-x^\ast$ satisfies the source condition $x^\dag-x^\ast=(F^\prime(x^\dag)^\ast F^\prime(x^\dag))^\nu v$ for some $v\in\mathcal{Y}$ and $0<\nu\leq\frac{1}{2}$.
\end{itemize}
Let $||v||$ be sufficiently small. Then, if the regularization parameter is stopped according to the discrepancy principle, i.e., at the unique index $k_\ast$ for which for the first time
\[
||F(x_{k})-y^\delta||\leq\hat\tau\delta
\]
with $\hat\tau>2\frac{1+\zeta}{1-2\zeta}>2$, we obtain
\begin{equation}
||x^\dag-x_{k_\ast}^\delta||\leq c||v||^{1/(2\nu+1)}\delta^{2\nu/(2\nu+1)}.
\end{equation}
\end{theorem}
We can obtain a stochastic version of Theorem \ref{thm:lw_nonlin_det} in the same way and with the same techniques used to show that Theorem \ref{thm:tikh_nonlin_stoch} followed from Theorem \ref{thm:tikh_nonlin_det}. 
\begin{theorem}\label{thm:lw_nonlin_stoch}
Let $\mathcal{D}(\mathcal{F})$ be convex, $y^\eta\in\mathcal{Y}$ be given with $\rho_K(y,y^\eta)$ and let $x^\dag(\omega)$ denote the $x^\ast$-minimum norm solution of \eqref{eq:problem}. Assume \eqref{eq:problem} has a solution in $\mathcal{B}_\vartheta(x^\ast(\omega))$ for almost all $\omega$. Furthermore let the following conditions hold on $\mathcal{B}_{2\vartheta}(x^\ast)$.
\begin{itemize}
\item[a)]$F^\prime(x,\omega)=R_{x,\omega}F^\prime(x^\dag(\omega),\omega)$  where for almost all $\omega$ the set $\{R_{x,\omega}:x\in\mathcal{B}_{\vartheta}(x^\ast)\}$ describes a family of bounded linear operators with 
\[
||R_{x,\omega}-I||\leq C(\omega)||x-x^\dag(\omega)||
\]
\item[b)]$x^\dag-x^\ast$ satisfies the source condition 
\[
x^\dag(\omega)-x^\ast(\omega)=(F^\prime(x^\dag(\omega),\omega)^\ast F^\prime(x^\dag(\omega),\omega))^\nu v(\omega)
\] for some $v(\omega)\in\mathcal{Y}$ and $0<\nu\leq\frac{1}{2}$.
\item[c)] $\mathbb{P}(\omega\in\Omega:C(\omega)||v(\omega)||>c\}<\varphi_{cl}(c)$
\item[d)] $\mathbb{P}(\omega\in\Omega: ||v(\omega)||>\tau)<\varphi(\tau)$
\end{itemize}
Then, if the regularization parameter is stopped according to the discrepancy principle, i.e., at the unique index $k_\ast$ for which for the first time
\[
||F(x_{k})-y^\eta||\leq\hat\tau\rho_K(y,y^\eta)
\]
with $\hat\tau>2$, we obtain for $c_0>0$ sufficiently small the rate
\begin{align*}
&\rho_K(x^\dag,x_{k_\ast}^\eta)\leq \\&\inf_{0<\tau\leq\infty}\max\left\lbrace \rho_K(y,y^\eta)+\varphi_{cl}(c_0)+\varphi_{de}(\tau),\tilde{c} \tau^{1/(2\nu+1)}\rho_K(y,y^\eta)^{2\nu/(2\nu+1)}\right\rbrace
\end{align*}
where the constant $\tilde{c}$ depends on $\nu$ only.
\end{theorem}


In the fully stochastic setting, the source condition b) from Theorem \ref{thm:lw_nonlin_stoch} need not hold with constant exponent $\nu$ for all $\omega\in\Omega$. There are at least two situations which lead to the power $\nu$ being a stochastic quantity as well, i.e., it holds
\begin{equation}\label{eq:sourcecond_stoch}
x^\dag(\omega)=(F^\prime(x^\dag(\omega),\omega)^\ast F^\prime(x^\dag(\omega),\omega))^{\nu(\omega)}v(\omega)
\end{equation}
with $0<\nu(\omega)\leq\frac{1}{2}$.

In the first case all solutions $x^\dag(\omega)$ come from some initial element $v(\omega)=v\in \mathcal{Y}$, with small $\mathcal{Y}$-norm. Some randomly smoothing operator is acting on this element and generates $x^\dag(\omega)$. (One could for instance think of some kind of evolution process, e.g., a diffusion process that is applied to some initial value $v$). The smoothness of $x^\dag(\omega)$ is therefore random.

Secondly, $x^\dag$ may be a deterministic element satisfying a certain smoothness condition. The data $y(\omega)$ is generated by applying a forward operator $F(\cdot,\omega)$ with random smoothness properties. If the realization of $F(\cdot,\omega)$ is strongly smoothing, this corresponds to a source condition with small $\nu(\omega)$, if $F(\cdot,\omega)$ is weakly smoothing we have the source condition with larger $\nu(\omega)$.

The following proposition shows the convergence rate that results from the source condition \eqref{eq:sourcecond_stoch} for the case that $\nu(\omega)$ is uniformly distributed on the interval $[0,\frac{1}{2}]$.

\begin{theorem}
Let all conditions of Theorem \ref{thm:lw_nonlin_stoch} hold except for b) and d). Let $x^\dag(\omega)$ satisfy \eqref{eq:sourcecond_stoch} where $||v(\omega)||$ is uniformly bounded and sufficiently small. Let $\nu(\omega)$ be uniformly distributed on the interval $[0,\frac{1}{2}]$, i.e.,
\[
\mathbb{P}\left(\omega\in\Omega: 0\leq \nu(\omega)<\nu\leq\frac{1}{2}\right)=2\nu.
\]
Then the approximations $x_{k_\ast}^\eta$ obtained by Landweber's method satisfy the convergence rate
\begin{equation}\label{eq:conv_rate_stoch_nu}
\rho_K(x^\dag,x_{k_\ast}^\eta)=\mathcal{O}\left(\frac{W(-\log(\rho_K(y,y^\delta))}{-\log(\rho_K(y,y^\delta)}\right)
\end{equation}
where $W$ denotes the Lambert W-function, defined by $W(z)e^{W(z)}=z$, see \cite{LAMW1996}.
\end{theorem}
\begin{proof}
As can be seen from the proof of Theorem 3.1 in \cite{HNS1995}, the requirement ``$||v||$ sufficiently small'', becomes stronger, the larger $\nu$ is. Supposing that $||v||$ in \eqref{eq:sourcecond_stoch} is sufficiently small for the case $\nu=\frac{1}{2}$, implies therefore that also the convergence conditions for $\nu\leq\frac{1}{2}$ are satisfied.

Secondly we observe that the convergence rate in Theorem \ref{thm:lw_nonlin_stoch} contains a constant $\tilde{c}$ that depends on $\nu$. Although it is difficult to state an explicit formula for $\tilde{c}$, investigation of \cite{HNS1995} shows, that $\tilde{c}(\nu)$ attains its maximum value when $\nu=\frac{1}{2}$. 

After these observations we start with the actual derivation of the convergence rate. For the sake of simplicity we assume that all appearing constants are just equal to 1. Furthermore we may assume that $\varphi_{cl}(\cdot)$ and $\varphi_{de}(\cdot)$ both vanish. Asymptotically, for given $\omega$ we therefore have the estimate
\[
||x^\dag(\omega)-x_{k_\ast}^\eta(\omega)||\leq \rho_K(y,y^\eta).
\]
Measuring the distance in the Ky Fan metric we must, since we assumed that $\nu(\omega)$ is as in \eqref{eq:sourcecond_stoch}, solve the equation
\begin{equation}
\rho_K(y,y^\delta)^{\frac{2\nu}{2\nu+1}}=2\nu
\end{equation}
for $\nu$. We first consider the simplified equation
\[
\rho_K(y,y^\delta)^{2\tilde\nu}=2\tilde\nu
\]
which is solved by
\[
\tilde\nu(\rho_K(y,y^\delta))=\frac{W(-\log \rho_K(y,y^\delta))}{-2\log \rho_K(y,y^\delta)}.
\]
In the following we show that the above approximate solution is sufficiently accurate. Therefore we construct a better estimate via the ansatz $\nu(\rho_K(y,y^\delta))=\tilde\nu(\rho_K(y,y^\delta))(1+\varepsilon(\tilde\nu(\rho_K(y,y^\delta)))$. The original equation then contains the term $2\tilde\nu+3\tilde\nu\varepsilon+1$. Neglecting the quadratic part, we can replace this term with $2\tilde\nu+1$, and obtain an equation that MATHEMATICA can solve for $\varepsilon(\tilde\nu)$. The solution for the correction term is given as
\[
\varepsilon(\tilde\nu)=\frac{\log(\tilde\nu)+(2\tilde\nu+1)W\left(-\frac{\log(\tilde\nu)}{2\tilde\nu^2+\tilde\nu}\right)}{-\log(\tilde\nu)}
\]
and tends to zero approximately linearly in $\tilde\nu$. Thus this correction becomes small rather quickly, and we can consider the asymptotic bound in \eqref{eq:conv_rate_stoch_nu} as sufficiently accurate due to the asymptotics of the Lambert W-function.
\end{proof}

\section{Examples}\label{sec:examples}
We will now apply the theory developed in the previous section to selected deterministic regularization methods.
\subsection{Filter-based regularization methods}
Let $A$ be a linear compact operator between Hilbert spaces $\mathcal{X}$ and $\mathcal{Y}$ with singular system $\{\sigma_n,u_n,v_n\}_{n\in\mathbb{N}}$, see e.g. \cite{EHN1996}. Then, for $y\in\mathcal{D}(A)$, the generalized inverse $A^\dag$ to $A$ is given by
\begin{equation}\label{eq:gen_inverse}
A^\dag y=\sum_{\sigma_n>0} \sigma_n^{-1}\langle y,u_n\rangle v_n. 
\end{equation}
Since for compact operators the singular values approach zero, their inverse blows up and the generalized inverse yields a meaningless solution to \eqref{eq:problem} for noisy data. A popular class of regularization methods is based on the filtering of the generalized inverse. Introducing an appropriate filter function $F_\alpha(\sigma)$ depending on the regularization parameter $\alpha$ that controls the growth of $\sigma^{-1}$, the regularized solutions are defined by
\begin{equation}\label{eqn:filtered_req}
R_\alpha(y)=\sum_{\sigma_n>0} F_\alpha(\sigma)\sigma_n^{-1}\langle y,u_n\rangle v_n.
\end{equation}
Examples for filter based methods are for example the classical Tikhonov regularization, truncated singular value decomposition or Landwebers method \cite{Louis,EHN1996}. 
The regularization properties are fully determined by the filter functions. In the deterministic setting, the conditions can be found in, e.g.,\cite[Theorem 3.3.3.]{Louis}. Convergence rates can be obtained for a priori and a posteriori parameter choice rules under stricter conditions on the filter functions. We will only comment on an a priori choice here in order to keep this section short. An example of the discrepancy principle as a posteriori parameter choice is given in the next section in a different context. Using the smoothness condition
\begin{equation}\label{eq:nunorm}
x^\dagger \in \textup{range}{(A^\ast A)^{\nu/2}},\quad
\|x^\dagger\|_\nu:= \{||z||_\mathcal{X}: x^\dag=(A^\ast A)^{\nu/2}z, z\in\mathcal{N}(A)^\perp\} \leq \varrho
\end{equation}
the following theorem can be obtained.
\begin{theorem}{\cite[Theorem 3.4.3]{Louis}}\label{thm:louis_filter_optimality}
Let $y\in\textup{range}(A)$ and $||y-y^\delta||_{\mathcal{Y}}\leq\delta$. Assume that it holds $||x^\dag||_\nu\leq\varrho$ and for $0\leq\nu\leq \nu^\ast$,
\begin{align}
&\sup_{0<\sigma\leq \sigma_1}\sigma^{-1}|F_{\alpha}(\sigma)|\leq c \alpha^{-\beta}\label{eq:filter_conditions1}\\
&\sup_{0<\sigma\leq \sigma_1}|1-F_{\alpha}(\sigma)|\sigma^{\nu^\ast} 
\leq c_{\nu^\ast} \alpha^{\beta {\nu^\ast}}\label{eq:filter_conditions2},
\end{align}
where $\beta>0$ and $c,c_{\nu^\ast}$ are constants independent of $\delta$. Then with the a priori parameter choice 
\begin{equation}\label{eq:apriori_filter}
\alpha=C\left(\frac{\delta}{\varrho}\right)^{1/\beta(\nu+1)},\qquad C>0 \quad \textrm{fixed},
\end{equation}
the method induced by the filter $F_{\alpha}$ is order optimal for all $0\leq\nu\leq\nu^{\ast}$, i.e.,
\[
||x^\dag-R_\alpha y^\delta||\leq c \delta^{\frac{\nu}{\nu+1}}\varrho^{\frac{1}{\nu+1}}
\]
for some constant $c$ independent of $\delta$ and $\varrho$.
\end{theorem}
Now we use Theorem \ref{thm:lifting_rates_kyfan} and obtain convergence rates in the Ky Fan metric.
\begin{theorem}
Let $y\in\textup{range}(A)$ and $\rho_K(y,y^\eta)$ be known. Assume that it holds $||x^\dag||_\nu\leq\varrho$ and for $0\leq\nu\leq \nu^\ast$, \eqref{eq:filter_conditions1} and \eqref{eq:filter_conditions2} hold. Then with the a priori parameter choice 
\begin{equation}\label{eq:apriori_filter_stoch}
\alpha=C\left(\frac{\rho_K(y,y^\eta)}{\varrho}\right)^{1/\beta(\nu+1)},\qquad C>0 \quad \textrm{fixed},
\end{equation}
the method induced by the filter $F_{\alpha}$ fulfills
\[
\rho_K(x^\dag,R_\alpha y^\eta)\leq c \rho_K(y,y^\eta)^{\frac{\nu}{\nu+1}}\varrho^{\frac{1}{\nu+1}}
\]
for some constant $c$ independent of $\delta$ and $\varrho$.
\end{theorem}
More about filter methods in the stochastic setting including numerical examples can be found in \cite{diss}.

\subsection{Sparsity-regularization for an autoconvolution problem}
We consider an autoconvolution equation
\begin{equation}\label{eq:autoconv}
[F(x)](s)=\int_0^s x(s-t)x(t)\, dt,\qquad 0\leq s\leq 1
\end{equation}
between Hilbert spaces $\mathcal{X}=L_2[0,1]$ and $\mathcal{Y}=L_2[0,1]$ where $x\in \mathcal{D}(F)$. Such an equation is of great interest in, for example, stochastics or spectroscopy and has been analyzed in detail in \cite{GorHof}. Recently, a more complicated autoconvolution problem has emerged from a novel method to characterize ultra-short laser pulses \cite{autoconv2014,autoconv2015}. Here, we want to show the transition from the deterministic setting to the stochastic setting in a numerical example. We base our results on the deterministic paper \cite{AnzRam10}.

Using the Haar-wavelet basis, the authors of \cite{AnzRam10} reformulate \eqref{eq:autoconv} as an equation from $\ell_2$ to $\ell_2$ by switching to the space of coefficients in the Haar basis. In order to stabilize the inversion, an $\ell_1$ penalty term is used such that the task is to minimize the functional
\begin{equation}\label{eq:j_alpha_nonlin}
J_\alpha(x)=||F(x)-y^\delta||_2^2+\alpha||x||_1.
\end{equation}
The regularization parameter $\alpha$ in \eqref{eq:j_alpha_nonlin} is chosen according to the discrepancy principle. In \cite{AnzRam10}, the following formulation is used: For $1<\tau_1\leq\tau_2$ choose $\alpha=\alpha(\delta,y^\delta)$ such that
\begin{equation}\label{eq:discrepancy_nonlin}
\tau_1\delta\leq ||F(x_\alpha^\delta)-y^\delta||_2\leq\tau_2\delta
\end{equation}
holds. The authors show that this leads to a convergence of the regularized solutions against a solution of \eqref{eq:autoconv} with minimal $\ell_1$-norm of its coefficients. It was also shown that the regularization parameter fulfills
\begin{equation}\label{eq:alpha_props}
\alpha(\delta,y^\delta)\rightarrow 0,\qquad \frac{\delta^2}{\alpha(\delta,y^\delta)}\rightarrow0\quad \text{as}\quad \delta\rightarrow 0.
\end{equation}

By courtesy of Stephan Anzengruber we were allowed to use the original code for the numerical simulation in \cite{AnzRam10}.   We only changed the parts directly connected to the data noise. Namely, we replaced the deterministic error $||y-y^\delta||_2\leq\delta$ with i.i.d Gaussian noise,
\[
y^\eta=y+\epsilon,
\]
$\epsilon\sim\mathcal{N}(0,\eta^2I)$. The discretization is due to the truncation of the expansion of the functions in the Haar-basis after $m$ elements. The parameter choice \eqref{eq:discrepancy_nonlin} was realized with $\delta$ replaced by $\tau(\eta)\mathbb{E}(||\epsilon||_2)$ in accordance with Theorem \ref{thm:lifting_convergence_kyfan}. Instead of the correct expectation
\[
\mathbb{E}(||\epsilon||_2)=\frac{\eta}{\sqrt{2}}\frac{\Gamma(\frac{m+1}{2})}{\Gamma(\frac{m}{2})},
\]
see \cite{diss}, we used the upper bound 
\[
\mathbb{E}(||\epsilon||_2)\leq \eta\sqrt{m}
\]
since, as shown in this chapter, the expectation has to be ``blown up'' anyway.  In a first experiment we let $\tau(\eta)=1.3=\text{const}$. In this case, the numerical results suggest that the regularization parameter decreases too fast, i.e., $\frac{(\tau(\eta)\mathbb{E}(||\epsilon||_2))^2}{\alpha}$ does not converge to zero as the requirement in \eqref{eq:alpha_props} states; see Figure \ref{fig:conv_nonlin}. For comparison, in a second run we chose $\tau(\eta)=\sqrt{1-\log(\eta^2 2\pi m^2(\frac{e}{2})^m)}$ where $m$ is the amount of data points. This way, $\tau(\eta)\mathbb{E}(||\epsilon||_2) \propto \rho_K(y,y^\eta)$. Now $\frac{(\tau(\eta)\mathbb{E}(||\epsilon||_2))^2}{\alpha}$ converges to zero as it should be the case from \ref{eq:alpha_props}, see Figure \ref{fig:conv_nonlin}.

At this point we would like to mention that the discrepancy principle using the Ky fan distance and the deterministic one are not completely equivalent since a different way of measuring the noise is used. Typically the stochastic noise level will be smaller (it need to bound 100\% of the possible realizations) and the iteration will be stopped later than in the deterministic setup.

\begin{figure}
\begin{center}
\includegraphics[width=0.49\linewidth]{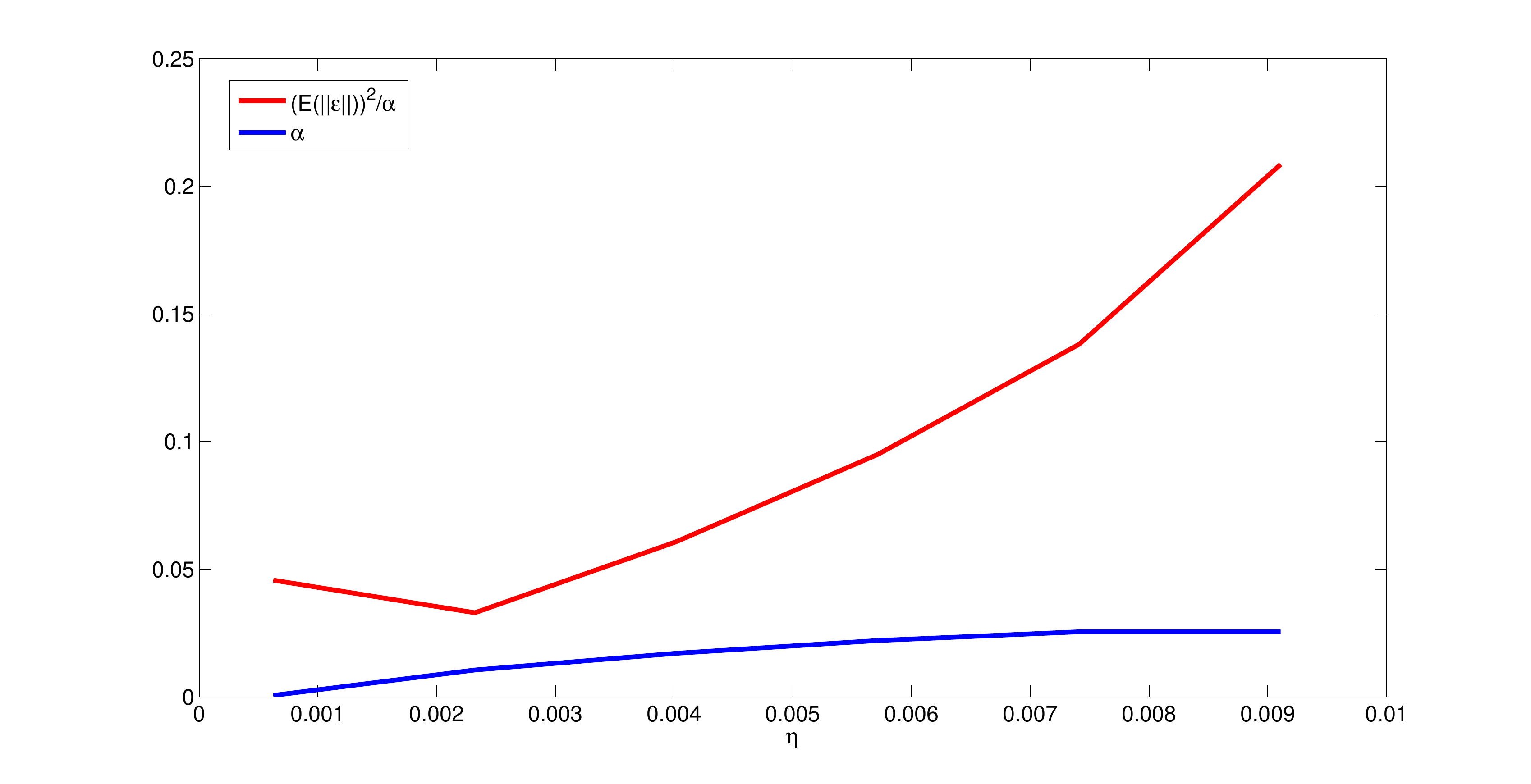}
\includegraphics[width=0.49\linewidth]{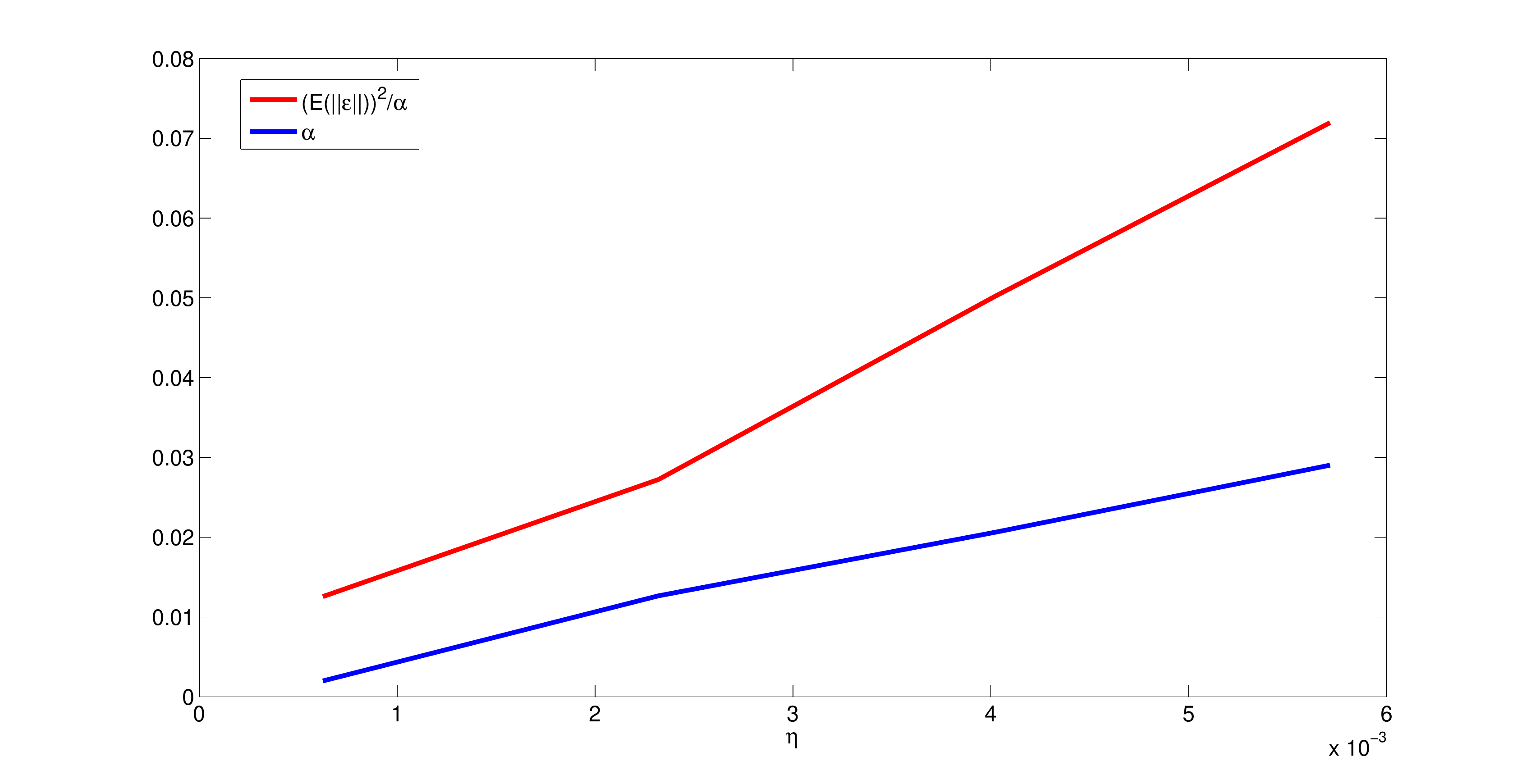}\end{center}\caption{Regularization error (red) and regularization parameter (blue) versus variance $\eta$. Left: constant value of $\tau$ in the discrepancy principle with the expectation of the noise leads to the regularization parameter decreasing too fast. Right: increasing $\tau$ appropriately with decreasing variance resolves this issue.}
\label{fig:conv_nonlin}\end{figure}

\subsection{Linear Inverse Problems with Besov-space prior}
In \cite{bayessparse2014} the lifting strategy was used in a slightly different way. In particular, the Ky Fan metric was used to obtain a novel parameter choice rule. The convergence rates obtained there, however, can also be viewed in the framework of this work. The scope of that paper was to transfer the deterministic convergence results from \cite{DDD04} into the stochastic setting. The seminal paper \cite{DDD04} initiated the investigation of sparsity-promoting regularization for Inverse Problems. Looking for the solution of the linear ill-posed problem 
\begin{equation}\label{eq:linear}
Ax=y
\end{equation}
between Hilbert spaces $\mathcal{X}$ and $\mathcal{Y}$ with given noisy data $y^\delta=y+\epsilon$, the regularization strategy was to obtain an approximation $x_\alpha^\delta$ to $x^\dag$ via
\begin{equation}\label{eq:func_det_ddd}
x_\alpha^\delta=\min_{x} ||Ax-y^\delta||_2^2+\sum_{\lambda\in\Lambda} w_\lambda |\langle x,\psi_\lambda \rangle|^p \psi_\lambda,
\end{equation}
where $\Lambda$ is an appropriate index set, $w_\lambda>0$ $\forall\lambda\in\Lambda$, $\{\psi_\lambda\}_{\lambda\in\Lambda}$ a dictionary (typically an orthonormal basis or frame) in $\mathcal{X}$ and $1\leq p\leq2$. Choosing a sufficiently smooth wavelet basis for $\{\psi_\lambda\}_{\lambda\in\Lambda}$ and setting $w_\lambda=2^{\zeta|\lambda|p}$ with $\zeta=s-d(\frac{1}{2}-\frac{1}{p})>0$, the penalty term in \eqref{eq:func_det_ddd} corresponds to a norm in the Besov space $B^s_{p,p}(\mathbb{R}^d)$. Formulating the problem of determining $x$ from noisy data $y^\eta=y+\epsilon$, $\epsilon\sim\mathcal{N}(0,\eta^2I_m)$, in the Bayesian setting with the distributions $\pi_\epsilon(y^\delta|x)\propto\exp(-\frac{1}{2\eta^2}||Ax-y^\delta||_2^2)$ and $\pi_{pr}(x)\propto\exp(-\frac{\tilde{\alpha}}{2}||x||^p_{B_{p,p}^s(\mathbb{R}^d)})$ and using the maximum a-posteriori solution lead to the formulation 
\begin{equation}
x^{\textup{MAP}}=\min_x ||Ax-y^\delta||_2^2+\tilde{\alpha}\eta^2 ||x||^p_{B^s_{p,p}(\mathbb{R}^d)}
\end{equation}
where $\eta$ is the variance of the noise and $\tilde{\alpha}$ can roughly be described as the inverse variance of the prior. The product $\alpha:=\tilde{\alpha}\eta^2$ gives the actual regularization parameter. In direct application of Theorem \ref{thm:lifting_convergence_kyfan}, the deterministic condition
\[
\alpha\rightarrow0,\quad \frac{\delta^2}{\alpha}\rightarrow0\textup{ as }\delta\rightarrow0,
\]
with $\delta$ replaced by $\rho_K(y,y^\eta)$ from \eqref{def:error_neupik} translates to the conditions
\[
\tilde{\alpha}\eta^2\rightarrow 0, \quad \frac{\log(\eta)}{\tilde{\alpha}}\rightarrow0\textup{ as }\eta\rightarrow0,
\]
leading to convergence of $x^{\textup{MAP}}$ to the unique (in case $p=1$ the operator is assumed to be injective) solution $x^\dag$ of minimal norm $||\cdot||_{B^s_{p,p}(\mathbb{R}^d)}$ in the Ky Fan metric. The proof of convergence rates is based on two assumption:
\[
C_l\sum_{\lambda\in\Lambda} 2^{-2|\lambda|\beta}|\langle x,\psi_\lambda\rangle|^p\leq ||Ax||\leq C_u\sum_{\lambda\in\Lambda} 2^{-2|\lambda|\beta}|\langle x,\psi_\lambda\rangle|^p
\]
where $\beta,C_l,C_u>0$ and
\[||x^\dag||_{B_{p,p}^s(\mathbb{R}^d)}\leq\rho
\]
for some $\rho>0$.
Combining Proposition 4.5, Proposition 4.6, Proposition 4.7 from \cite{DDD04} it is
\begin{equation}\label{eq:bayessparse_convrate_det}
||x_\alpha^\delta-x^\dag||\leq C\left( \delta+\sqrt{\delta^2+\alpha\rho^p} \right)^{\frac{\zeta}{\zeta+\beta}}\left(\rho+\left(\rho^p+\frac{\delta^2}{\alpha}\right)^{1/p}\right)^{\frac{\beta}{\zeta+\beta}}.
\end{equation}
Translated into the stochastic setting, the right hand side of \eqref{eq:bayessparse_convrate_det} reads
\begin{equation}\label{eq:ass1}
C\mathcal{E}(\eta,m,\tilde{\alpha})^{\frac{\zeta}{\zeta+\beta}}\tilde{\rho}^{\frac{\beta}{\zeta+\beta}}
\end{equation}
where with $L_m(\eta)=\min\{0,\eta^22\pi m^2(\frac{e}{2})^m\}$,
\begin{equation}\label{eq:mathcal_e}
\mathcal{E}(\eta,m,\tilde{\alpha}):=\eta \left(\sqrt{m-L_m(\eta)}+\sqrt{m-L_m(\eta)+\frac{\tilde{\alpha}\rho^p}{2}}\right)
\end{equation}
and
\[
\tilde{\rho}=\rho+\left(\rho^p+\frac{2m-L_m(\eta)}{\tilde{\alpha}}\right)^{1/p}.
\]
We know that the deterministic rate is an upper bound to the reconstruction error whenever $||y-y^\eta||=||\epsilon||\leq \rho_K(y,y^\eta)$ and $||x^\dag||_{B_{p,p}^s(\mathbb{R}^d)}\leq\rho$. Hence, it is
\begin{equation}\label{eq:bayessparse_convrate_stoch}
\mathbb{P}\left(||x^{\textup{MAP}}-x^\dag||\geq C\mathcal{E}(\eta,m,\tilde{\alpha})^{\frac{\zeta}{\zeta+\beta}}\tilde{\rho}^{\frac{\beta}{\zeta+\beta}}\right)\leq \frac{\Gamma( \frac{m}{2},m-L_m(\eta))}{\Gamma( \frac{m}{2})}+\frac{\Gamma(\frac{n}{p},\frac{\tilde{\alpha}\varrho^p}{2})}{\Gamma(\frac{n}{p})}
\end{equation}
where
\[
\mathbb{P}(||y-y^\eta||>\rho_K(y,y^\eta))=\frac{\Gamma( \frac{m}{2},m-L_m(\eta))}{\Gamma( \frac{m}{2})}.
\]
and
\[
\mathbb{P}(||x^\dag||_{B_{p,p}^s(\mathbb{R}^d)}\geq\rho)=\frac{\Gamma(\frac{n}{p},\frac{\tilde{\alpha}\varrho^p}{2})}{\Gamma(\frac{n}{p})}
\]
where the Besov-space functions were truncated after the first $n$ basis functions. By Definition of the Ky Fan metric, it follows immediately from \eqref{eq:bayessparse_convrate_stoch} that
\begin{equation}\label{eq:conv_rate_bayessparse_max}
\rho_K(x^{\textup{MAP}})=\max\left\lbrace C\mathcal{E}(\eta,m,\tilde{\alpha})^{\frac{\zeta}{\zeta+\beta}}\tilde{\rho}^{\frac{\beta}{\zeta+\beta}},\frac{\Gamma( \frac{m}{2},m-L_m(\eta))}{\Gamma( \frac{m}{2})}+\frac{\Gamma(\frac{n}{p},\frac{\tilde{\alpha}\varrho^p}{2})}{\Gamma(\frac{n}{p})}\right\rbrace.
\end{equation}
Since $\tilde{\alpha}$ is a free parameter, we can balance the terms in \eqref{eq:conv_rate_bayessparse_max}, i.e. solve the nonlinear equation
\[
C\mathcal{E}(\eta,m,\tilde{\alpha})^{\frac{\zeta}{\zeta+\beta}}\tilde{\rho}^{\frac{\beta}{\zeta+\beta}}=\frac{\Gamma( \frac{m}{2},m-L_m(\eta))}{\Gamma( \frac{m}{2})}+\frac{\Gamma(\frac{n}{p},\frac{\tilde{\alpha}\varrho^p}{2})}{\Gamma(\frac{n}{p})}
\]
for $\tilde{\alpha}$. With this parameter choice rule one obtains by construction 
\begin{equation}\label{eq:conv_rate_bayessparse_max2}
\rho_K(x^{\textup{MAP}})=\mathcal{O}(\mathcal{E}(\eta,m,\tilde{\alpha})^{\frac{\zeta}{\zeta+\beta}}\tilde{\rho}^{\frac{\beta}{\zeta+\beta}}).
\end{equation}

We can also apply the theory developed in this work to this problem. In the deterministic setting, see \cite{DDD04}, it was proposed to chose the regularization parameter $\alpha=\delta^2/\varrho^p$. Combining \cite[Proposition 4.5]{DDD04} and \cite[Proposition 4.7]{DDD04} then yields the rate
\[
||x_\alpha^\delta-x^\dag||\leq C \left(\frac{\delta}{C_l} \right)^{\frac{\varsigma}{\varsigma+\beta}}\varrho^{\frac{\beta}{\varsigma+\beta}}
\]
with $C_l$ from \eqref{eq:ass1} and some $C>0$. Theorem \ref{thm:lifting_rates_kyfan} then yields in the stochastic setting the parameter choice
$\alpha\sim\rho_K(y^\eta,y)^2/\varrho^p$ and
\begin{equation}\label{eq:conv_rate_bayes_sparse_lifted}
\rho_K(x_\alpha^\eta,x^\dag)=\mathcal{O}\left(\left(\frac{\rho_K(y^\eta,y)}{C_l} \right)^{\frac{\varsigma}{\varsigma+\beta}}\varrho^{\frac{\beta}{\varsigma+\beta}}\right).
\end{equation}
In the notation of \eqref{eq:mathcal_e} it is for Gaussian noise $\epsilon\sim\mathcal{N}(0,\eta^2 I_m)$
\begin{equation}\label{eq:rho_K_e}
\rho_K(y^\eta,y)\leq \eta\sqrt{m-L_m(\eta)},
\end{equation}
see Proposition \eqref{prop:kyfandistance}. Since $\rho_K(y^\eta,y)<\mathcal{E}(\eta,m,\tilde{\alpha})$, compare \eqref{eq:mathcal_e} and \eqref{eq:rho_K_e}, the convergence rate in \eqref{eq:conv_rate_bayessparse_max2} is slightly slower than the one in \eqref{eq:conv_rate_bayes_sparse_lifted}, but they share the same order of convergence.

\section*{Conclusions}
Our goal was to demonstrate how convergence results for Inverse Problems in the deterministic setting can be carried over into the stochastic setting. Using the Ky Fan metric, we have shown that, when only the noise is assumed to be stochastic whereas the other quantities are deterministic, this is is possible in a straight-forward way. Namely, assuming the knowledge of an estimate of $\rho_K(y,y^\eta)$, the convergence results and parameter choice follows from the deterministic setting by replacing $\delta$, which originates from the basic deterministic assumption $||y-y^\delta||~\leq~\delta$, with $\rho_K(y,y^\delta)$. We have shown that, under some slight modifications, it is possible to use the expectation as measure for the magnitude of the noise. In a fully stochastic situation, where additionally to the noise other objects might be of stochastic nature, the lifting of deterministic convergence results is possible as well. However, careful analysis is necessary in order to carry the deterministic conditions over into the stochastic setting.


\end{document}